\documentclass[10pt, a4paper]{scrartcl} 
\usepackage[utf8]{inputenc}
\usepackage[left=2.5cm,right=2.5cm,top=1.8cm,bottom=1.8cm]{geometry}
\usepackage{amsmath,amssymb,amsthm,mathtools}
\usepackage{xcolor}
\usepackage{graphicx}
\usepackage{bbm,dsfont}
\usepackage{hyperref} 
\usepackage{url}
\usepackage{framed}
\usepackage{tikz}
\usetikzlibrary{arrows.meta}
\usetikzlibrary{positioning}
\usetikzlibrary{decorations.markings}
\usetikzlibrary{shapes.misc}
\usetikzlibrary{hobby}
\usepackage{subfigure}
\usepackage[small,bf,sf]{caption}
\usepackage{mathrsfs}
\usepackage{esint}
\usepackage{csquotes}
\usepackage[english]{babel}
\usepackage[normalem]{ulem}
\PassOptionsToPackage{hyperfootnotes=false}{hyperref}
\usepackage{caption}
\usepackage[nameinlink,capitalise,sort]{cleveref}
\hypersetup{
  colorlinks = true,
  linkcolor = blue,
  urlcolor = blue,
  citecolor = blue
}
\usepackage{comment}

\usepackage[nameinlink,capitalise,sort]{cleveref}
\crefname{equation}{}{}
\crefname{enumi}{}{}

\crefname{figure}{Figure}{Figures}

\usepackage{enumerate}
\usepackage{enumitem}
\usepackage{wrapfig}

\theoremstyle{plain}
\begingroup
\theoremstyle{plain}
\newtheorem{theorem}{Theorem}[section]
\newtheorem{corollary}[theorem]{Corollary}
\newtheorem{proposition}[theorem]{Proposition}
\newtheorem{lemma}[theorem]{Lemma}
\theoremstyle{definition}
\newtheorem{definition}[theorem]{Definition}
\theoremstyle{assumptions}
\newtheorem{assumptions}[theorem]{Assumption}
\theoremstyle{remark}
\newtheorem{remark}[theorem]{Remark}

\endgroup

\theoremstyle{definition}
\theoremstyle{remark}

\numberwithin{equation}{section}

\newcommand{\R}{\mathbb{R}}

\begin{document}
\title{Conservation Laws with Discontinuous Flux:\\ A First Application to Traffic Flow}
\author{Roberta Bianchini, Maya Briani, Benedetto Piccoli}
\date{}
\maketitle

\begin{abstract}
    This work provides a traffic flow interpretation of the model recently introduced in \cite{ABS2025}. We adapt the conservation law with discontinuous gradient-dependent flux to a macroscopic traffic setting, assuming concave flux functions $f(u)$ and $g(u)$ defined on the density domain $[0,u_{\max}]$. To prevent violations of the maximal density constraint, we introduce modified Riemann Solvers tailored to this setting and generalize the existence theory for weak solutions developed in \cite{ABS2025}. Finally, we construct Riemann Solvers for a broader two-phase model incorporating a free-flow regime, in which $f(u)=g(u)$ on a subset of $[0,u_{\max}]$, and establish existence of weak solutions in this more general framework.
\end{abstract}

\section*{Introduction}
Modeling of vehicular traffic dynamics has a long history \cite{garavello2016models,helbing2001traffic,piccoli2009vehicular,zhang2024car}, still new phenomena emerging from
next-generation data
\cite{coifman2017critical,ji2026scalable}
remain elusive.
At the macroscopic level, the foundational model for traffic flow, based on scalar conservation laws, is the celebrated Lighthill–Whitham–Richards (LWR) model $u_t + \partial_x f(u) = 0$, where $u(t,x) \in [0, u_{\max}]$ denotes the vehicle density and $f(u) = u v(u)$ is the flow rate, with the average speed $v$ depending only on the density $u$.
The LWR model
captures the propagation of shock fronts and rarefaction waves, 
but fails to represent
more complex, non-equilibrium phenomena.
One of the most challenging is traffic hysteresis, which means that traffic states follow distinct paths during acceleration and deceleration cycles. 
Drivers tend to decelerate more cautiously when entering a congestion front than they accelerate when 
exiting towards free flow.
Hysteresis calls for the use of multivalued fundamental diagrams and hysteretic loops in the flow–density plane
and may explain the emergence of phantom traffic jams and stop-and-go waves.

Treiterer and Myers first detected traffic hysteresis \cite{treiterer1974hysteresis} by identifying clockwise hysteresis loops in aerial vehicle trajectory data.
Transportation engineering models linked hysteresis to asymmetric driver behavior, delayed reaction times, and anticipation \cite{zhang1999mathematical}, as well as multi-lane dynamics and headway adaptation during lane-changing maneuvers \cite{laval2011hysteresis}. 
To capture this microscopic phenomenon at the macroscopic scale, non-equilibrium dynamics prompted the development of second-order hyperbolic systems, most notably the Aw--Rascle--Zhang (ARZ) model \cite{aw2000resurrection,zhang2002non}.
Corli and Fan \cite{corli2024hysteretic} recently formulated the \emph{Hysteretic Aw--Rascle--Zhang (HARZ)} model to
reproduce persistent stop-and-go waves driven by hysteresis. They augmented the classical system with a hysteretic internal state variable and identified criteria dictating whether traffic oscillations persist or dissipate \cite{corli2019hysteresis, corli2024hysteretic}.

A parallel, mathematically compelling approach aims to capture hysteretic phase transitions while preserving the structural simplicity of the LWR model.
Amadori, Bressan, and Shen
\cite{ABS2025} introduced a scalar conservation law with a discontinuous, gradient-dependent flux:
\begin{equation}\label{eq:original}
    u_t + [\theta(u_x) f(u) + (1-\theta (u_x))g(u)]_x=0, \quad \theta(s)=\begin{cases}
        1, \quad s>0,\\
        0, \quad s<0. 
    \end{cases}
\end{equation}
Here $f,g \in C^{2}([0, u_{\max}])$ are strictly concave functions, where $u_{\max}>0$ is the maximal density.
The governing flux dynamically switches based on the sign of the local spatial derivative $u_x$, selecting the upper branch $f(u)$ in regions of compression or deceleration (where the density increases, $u_x > 0$) and the lower branch $g(u)$ in regions of expansion or acceleration ($u_x < 0$). 
Motivated by modeling of traffic hysteresis and following \cite{ABS2025}, we consider the \emph{unstable} condition:
\[f(u)>g(u) \quad u \in (0, u_{\max}).\]

The authors of \cite{ABS2025} established the existence of global weak entropy solutions for flux functions defined across the unbounded real line $\mathbb{R}$, with initial data which are piecewise monotone, i.e., increasing or decreasing on
a finite number of intervals.
Motivated by traffic modeling, we consider the strict enforcement of physical domain constraints, namely that density remains non-negative and bounded by $u_{\max}$. Moreover, traffic flow naturally exhibits a free-flow regime at low densities, say $0\leq u\leq u_{\mathrm{free}}$ for some 
$u_{\mathrm{free}}>0$,
with minimal vehicle interactions and coinciding acceleration and deceleration branches, i.e.,  $f(u) = g(u)$ for $u \in [0, u_{\mathrm{free}}]$.

In this paper, we first develop a complete study of Riemann solvers on the compact domain $[0, u_{\max}]$, designing boundary-adapted interface configurations. 
This choice generates new wave patterns that saturate at the maximal density $u_{\mathrm{max}}$ and exhibit vacuum-like behavior where the density vanishes, as in gas dynamics.
The evolution of these discontinuity interfaces is given by ordinary differential equations with discontinuous right-hand sides. We employ directional bounded variation theory \cite{bressan1988} and Picard fixed-point arguments to establish the existence of weak solutions for bounded variation, piecewise monotone initial data. 
Then, we generalize the framework to a two-phase traffic model incorporating an identical free-flow regime alongside a hysteretic congested regime, classifying all wave transitions across the critical density $u_{\mathrm{free}}$ and proving local-in-time existence.

\subsection*{Plan of the paper.} In the first section, we introduce our setting, including the assumptions on the flux functions and the initial data, as well as the notion of a weak solution. In Section 2, we describe and analyze the building blocks of the existence theory, namely the Riemann Solvers. Section 3 is devoted to the construction of weak solutions. Finally, in Section 4, we discuss the two-phase model with a free-flow regime.

\medskip

\section{Preliminaries}
Consider two concave functions $f, g$ that satisfy the following.
\begin{assumptions}\label{assump}
Let $u_{\max}>0$ denote the maximal density. Assume that
\[
f, g\in C^2([0,u_{\max}])
\]
are two strictly concave flux functions satisfying
\[
f(0)=g(0)=f(u_{\max})=g(u_{\max})=0.
\]
\end{assumptions}
We consider solutions to the Cauchy problem associated with the equation \eqref{eq:original} for the following class of initial data as in \cite{ABS2025}:
\begin{itemize}
    \item[(H1)] An initial profile
    \[
    u(0,x) = \bar{u}(x),
    \]
    where $\bar{u} \in L^\infty(\mathbb{R})$ is a piecewise monotone function with  $0\leq \bar u \leq u_{\max}$.
    \item[(H2)] A piecewise constant function $\bar{\theta} : \mathbb{R} \to \{0,1\}$, and a \emph{finite} set of interfaces
    \[
    \bar{y}_1 < \bar{y}_2 < \cdots < \bar{y}_N,
    \]
    associated with $\bar{u}$ and $\bar{\theta}$, as in the following definition.
\end{itemize}
\begin{definition}[\cite{ABS2025}]\label{def:interface}
        For a given couple $(\bar u(x), \bar\theta(x))$, a set of points 
        \[\bar{y}_1 < \bar{y}_2 < \cdots < \bar{y}_N,
        \]
        is a \emph{set of interfaces} for $(\bar u, \bar \theta)$ if it contains all the jumps in $\bar \theta$ and, for any open interval $J_i=(y_i, y_{i+1}), i \in \{0, \cdots, N\}$:
        \begin{itemize}
            \item If $\bar \theta(x)=1$ on $J_i$, then $\bar u(x)$ is monotone increasing.
            \item If $\bar \theta(x)=0$ on $J_i$, then $\bar u(x)$ is monotone decreasing.
        \end{itemize}
\end{definition}
Following \cite{ABS2025}, we introduce the following definition of \emph{weak solutions}.
\begin{definition}\label{def:weak_sol}
A BV function is a \emph{weak solution} to \eqref{eq:original} with initial data as in (H1)-(H2) if:
\begin{itemize}
    \item The map $t \to u(t, \cdot)$ is continuous from $[0, T]$ to $L^1_{loc}(\mathbb R)$,  satisfies the initial conditions, and $0\leq u(t,x) \leq u_{\max}$.
    \item There exists a Lebesgue measurable function $\theta (t, x):[0, T] \times \mathbb R \to [0, 1]$ such that
    \[u_x(t, \cdot) = \chi_{\theta=1}[u_x(t, \cdot)]_+ - \chi_{\theta=0}[u_x(t, \cdot)]_- \quad \text{for a.e.} \; t \in [0, T]. \]
    \item For every compactly supported test function $\varphi (t, x) \in C_c((0, T) \times \mathbb R)$:
    \[\int_0^T \int_{\mathbb R} u \varphi_t + (\theta f(u) + (1-\theta) g(u)) \varphi_x \, dx \, dt=0.\]
\end{itemize}
\end{definition}

\section{Riemann Solver}\label{sec:solvers}
We now present a general strategy to solve problems with piecewise constant initial data:
\begin{equation}\label{eq:rp}
\begin{array}{cc}
u(0,x) = \left\{\begin{array}{lcr}
u^- & \mbox{if} & x<0,
\\
u^+ & \mbox{if} & x>0,
\end{array}\right.
&
\theta(0,x) = \left\{\begin{array}{lcr}
\theta^- & \mbox{if} & x<0,
\\
\theta^+ & \mbox{if} & x>0.
\end{array}\right.
\end{array}
\end{equation}

We first analyze the admissibility of downward and upward jump discontinuities. 
\begin{lemma}[Downward jumps are not admissible]
    Let the flux functions $f,g$ satisfy the Assumption \ref{assump}. Let $(u^-,\theta^-)$ and $(u^+,\theta^+)$ denote the left and right state of a jump. Then, any downward jump ($u^->u^+$) violates the Lax admissibility condition
    \begin{equation}\label{eq:LaxCond}
        \theta^+f^\prime(u^+)+(1-\theta^+)g^\prime(u^+)\leq \lambda\leq \theta^- f^\prime(u^-)+(1-\theta^-)g^\prime(u^-),
    \end{equation}
    where
    \begin{equation}\label{eq:lambda}
        \lambda = \displaystyle\frac{[\theta^+ f(u^+)+(1-\theta^+)g(u^+)]-[\theta^- f(u^-)+(1-\theta^-)g(u^-)]}{u^+-u^-}.        
    \end{equation}
\end{lemma}
\begin{proof}
We distinguish the following three cases:
\begin{itemize}
    \item If $\theta^-=\theta^+=\theta \in \{0, 1\}$, then by the concavity of the flow functions
    $$
	\lambda = \theta\displaystyle\frac{f(u^+)-f(u^-)}{u^+-u^-}+(1-\theta)\displaystyle\frac{g(u^+)-g(u^-)}{u^+-u^-}  < \theta\ f^\prime(u^+) + (1-\theta)\ g^\prime(u^+),    
    $$ 
    which obviously violates the Lax condition.
    \item If $\theta^-=1$ and $\theta^+=0$, then $u^-$ (and $u^+$) is on the graph of $f$ (and $g$) respectively, and the Lax condition writes $g^\prime(u^+)\leq\lambda\leq f^\prime(u^-)$. By Assumption \ref{assump}, the concavity property and $u^+<u^-$, we obtain  
    $$
    		\lambda = \frac{f(u^-)-g(u^+)}{u^--u^+} = \frac{f(u^-)-f(u^+)}{u^--u^+}+\frac{f(u^+)-g(u^+)}{u^--u^+} > f^\prime(u^-), 
    $$
    which violates the Lax condition.
    \item If $\theta^-=0$ and $\theta^+=1$, then $u^-$ (and $u^+$) is on the graph of $g$ (and $f$) respectively, and the Lax condition writes $f^\prime(u^+)\leq\lambda\leq g^\prime(u^-)$. As before, by the Assumption \ref{assump}, the concavity property and $u^+<u^-$,
    $$
    		\lambda = \frac{f(u^+)-g(u^-)}{u^+-u^-} = \frac{f(u^+)-f(u^-)}{u^+-u^-}+\frac{f(u^-)-g(u^-)}{u^+-u^-} <  f^\prime(u^+),
    $$
    which, once again, violates the Lax condition.
\end{itemize}
\end{proof}

\begin{lemma}[When upward jumps are admissible]\label{lemma:up_jump}
Let $u^-<u^+$ and $\theta^-\ne\theta^+$. According to \eqref{eq:LaxCond}, upward jumps are Lax-admissible in the following two situations:
\begin{itemize}
\item $\theta^-=1$, $\theta^+=0$ and $u^* \leq u^+$, where $(u^*,g(u^*))$ denotes the unique point satisfying: 
\begin{equation}\label{eq:ustar}
u^*> u^- \quad\mbox{such that}\quad g^\prime(u^*) = \frac{f(u^-)-g(u^*)}{u^--u^*}.
\end{equation}
\item $\theta^-=0$, $\theta^+=1$ and $u^- \leq v^*$, where $(v^*,g(v^*))$ denotes the unique point satisfying:
\begin{equation}\label{eq:vstar}
v^* < u^+ \quad\mbox{such that}\quad g^\prime(v^*) = \frac{f(u^+)-g(v^*)}{u^+-v^*}.
\end{equation}
\end{itemize} 
\end{lemma}
\begin{proof}
\item Consider $\theta^-=1$ and $\theta^+=0$. By Assumptions \ref{assump} and recalling that $u^-<u^+$, we have  
    $$
    		\lambda = \frac{g(u^+)-f(u^-)}{u^+-u^-} = \frac{g(u^+)-f(u^+)}{u^+-u^-}+\frac{f(u^+)-f(u^-)}{u^+-u^-} < f^\prime(u^-). 
    $$
    On the other hand, for $u^-\leq u^* < u^+$,
    $$\begin{array}{ll}
    		\lambda &= \displaystyle\frac{g(u^+)-f(u^-)}{u^+-u^-} = \frac{g(u^+)-g(u^*)}{u^+-u^*}\frac{u^+-u^*}{u^+-u^-} + \frac{g(u^*)-f(u^-)}{u^*-u^-}\frac{u^*-u^-}{u^+-u^-}  \\
    		\bigskip
    		&> g^\prime(u^+) \displaystyle\frac{u^+-u^*}{u^+-u^-} + g^\prime(u^*) \frac{u^*-u^-}{u^+-u^-} > g^\prime(u^+), 
    \end{array}
    $$
    where in the last inequality we used the monotonicity of $g^\prime$, i.e.  when  $u^* < u^+$ $\Rightarrow$ $g^\prime(u^*) > g^\prime(u^+)$.
    
    \item For $\theta^-=0$ and $\theta^+=1$, the proof is similar and therefore we omit it.
\end{proof}
\begin{remark}[Rarefaction waves are emanated only from $g$] \label{remark:rar}
Rarefaction waves occur only for the flux $g$. Indeed, if the left state exceeds the right state ($u^- > u^+$), the solution within a rarefaction fan is monotone decreasing, i.e., $u_x < 0$ for all $(x,t)$. Consequently, $\theta(u_x(t,x)) = 0$ throughout the fan. 

On the other hand, for a rarefaction wave to arise for the flux $f$, one would require $\theta(u_x(t,x)) \equiv 1$ in the region, which is not possible in this setting.
\end{remark}

\medskip

In Assumption~\ref{assump} we allow
$f'(0)\neq g'(0)$ and $f'(u_{\max})\neq g'(u_{\max})$. Consequently, the
two points $u^*$ and $v^*$ identified in Lemma \ref{lemma:up_jump} may lie
outside the interval $[0,u_{\max}]$. More precisely, there exist two
limit points $u^\mp_{\mathrm{lim}}$ such that, beyond these thresholds,
the points $u^*$ and $v^*$ can no longer be found in $[0,u_{\max}]$.
This property is formalized in the following result.


\begin{lemma}
Assume that $f'(u_{\max})\neq g'(u_{\max})$. Then the tangent point $u^*$ defined in \eqref{eq:ustar} exists if and only if $u^-<u^-_{\mathrm{lim}}$, where $u^-_{\mathrm{lim}}$ is defined by
\begin{equation}\label{eq:umenolim}
g'(u_{\max})
=
\frac{f(u^-_{\mathrm{lim}})-g(u_{\max})}{u^-_{\mathrm{lim}}-u_{\max}}
=
\frac{f(u^-_{\mathrm{lim}})}{u^-_{\mathrm{lim}}-u_{\max}}.
\end{equation}

Similarly, assume that $f'(0)\neq g'(0)$. Then the tangent point $v^*$ defined in \eqref{eq:vstar} exists if and only if $u^+>u^+_{\mathrm{lim}}$, where $u^+_{\mathrm{lim}}$ is defined by
\begin{equation}\label{eq:upiulim}
g'(0)
=
\frac{f(u^+_{\mathrm{lim}})-g(0)}{u^+_{\mathrm{lim}}}
=
\frac{f(u^+_{\mathrm{lim}})}{u^+_{\mathrm{lim}}}.
\end{equation}
If, however, $f'(0)=g'(0)$ and $f'(u_{\max})=g'(u_{\max})$, then
$u^+_{\mathrm{lim}}=0$ and $u^-_{\mathrm{lim}}=u_{\max}$.
\end{lemma}
\begin{proof}
The statement follows from the monotonicity of secant slopes for concave functions.
\end{proof}

In the description of the Riemann solvers below, we distinguish the cases $u^-<u^-_\mathrm{lim}$ and $u^-\geq u^-_\mathrm{lim}$,  $u^+\leq u^+_\mathrm{lim}$ and $u^+> u^+_\mathrm{lim}$.

\subsection{Case 1: $\theta^-=0$, $\theta^+=0$ ($g\rightarrow g$).}\label{sec:case1}
    For any pair $(u^-,u^+)$, a solution to the Riemann problem is obtained by letting $u=u(t,x)$ be the solution to the scalar conservation law
    \begin{equation}\label{eq:CLg}
    u_t+g(u)_x = 0,
    \end{equation}
    with initial data \eqref{eq:rp}. Two cases must be distinguished.
\begin{enumerate}[label=(\alph*)]
    \item $u^-<u^+$. In this case, the solution $u$ consists of a single
    upward jump satisfying the Lax condition \eqref{eq:LaxCond}. Since
    $u$ is increasing across an upward jump, the jump represents an
    interface, and we set
    \begin{equation}\label{eq:theta1a}
        \theta(t,x)=\left\{\begin{array}{cc}
           0  & x\neq\lambda t,\\
           1  & x=\lambda t,
        \end{array}\right.
    \end{equation}
    where $\lambda=(g(u^+)-g(u^-))/(u^+-u^-)$ is the
    Rankine--Hugoniot speed of the shock. This choice is consistent with
    the fact that $\theta(u_x)$ is undetermined where $u_x=0$ and can be
    chosen arbitrarily on intervals where $u$ is constant.
    
    \item $u^-\geq u^+$. In this case, the solution contains a centered
    rarefaction fan and there is no interface. We therefore set
    \begin{equation}\label{eq:theta1b}
        \theta(t,x)=0 \mbox{ for all } t,x.
    \end{equation}
\end{enumerate}
\subsection{Case 2: $\theta^-=1$, $\theta^+=0$ ($f\rightarrow g$).}\label{sec:case2}

    Recall that $u^*$ in \eqref{eq:ustar} is the unique tangent point satisfying $u^*>u^-$, namely
    \begin{equation}\label{eq:tangent-point}
    g^\prime(u^*) = \frac{f(u^-)-g(u^*)}{u^--u^*}, \quad u^- < u^*.
    \end{equation}
    In the general case $f'(u_{\max})\neq g'(u_{\max})$, the tangent point exists if and only if $u^->u^-_\mathrm{lim}$ where $u^-_\mathrm{lim}$ is defined in \eqref{eq:umenolim}. 
    We therefore consider $u^-< u_{\mathrm{lim}}$ and $u^-\geq u_{\mathrm{lim}}$ separately.
    
    \begin{enumerate}
    \item[(A)] If $u^-<u^-_{\mathrm{lim}}$, $u^- < u^* \leq u^+$, then the solution is a single jump from $u^-$ to $u^+$ with speed
    $$
    \lambda = \displaystyle\frac{f(u^-)-g(u^+)}{u^--u^+}, \quad u^--u^+ < 0.
    $$
    By Lemma~\ref{lemma:up_jump}, this jump is Lax admissible. In this case the location of the jump is an interface, and we have
    \[
        (u,\theta)(t,x) = \begin{cases}
              (u^-,1) &  x/t < \lambda \\
              (u^+,0) &  x/t > \lambda.
        \end{cases} 
    \]

    \item[(B)] If $u^- < u^-_{\mathrm{lim}}$, $u^- < u^*$, and $u^+ < u^*$, recalling from Lemma \ref{lemma:up_jump} that a direct jump is not admissible if $u^- < u^+ < u^*$, the solution consists of a shock connecting $(u^-,\theta^- = 1)$ to $(u^*,\theta^+ = 0)$, with speed
    \[
    \lambda = g^\prime(u^*).
    \] 
    This is followed by a rarefaction fan associated with the flux $g$, connecting $g^\prime(u^*)$ to $g^\prime(u^+)$ where $g^\prime(u^*) < g^\prime(u^+)$. 

    In summary, the solution consists of two waves, a shock followed by a rarefaction:
    \begin{equation}\label{eq:sol2b}
    (u,\theta)(t,x) =
    \begin{cases}
    (u^-,1) & \text{if } x/t < \lambda=g^\prime(u^*), \\
    (w,0) & \text{if } x/t=g'(w), \ g^\prime(u^*) < x/t < g^\prime(u^+), \\
    (u^+,0) & \text{if } x/t > g^\prime(u^+).
    \end{cases}
    \end{equation}
   Notice that, also in this case, the location of the jump represents an interface.
    
    \item[(C)] If $u^- \geq u^-_{\mathrm{lim}}$, then the value $u^*$ is not defined. In this case, $u$ consists of a shock connecting $u^-$ to $u_{\max}$, followed by a rarefaction fan (associated with the flux $g$) connecting $u_{\max}$ to $u^+$. The speed of the shock is then given by 
    \[
    \lambda = \frac{f(u_{\max}) - f(u^-)}{u_{\max} - u^-},
    \]
    which is Lax admissible.
    Since $f$ is concave and vanishes at $u_{\max}$, the slope of the secant line connecting any point $u < u_{\max}$ to $u_{\max}$ decreases strictly as $u$ approaches $u_{\max}$ from the left. Given that $0 < u^-_{\mathrm{lim}} < u^- < u_{\max}$, it follows that
    \[
    \lambda < \frac{f(u_{\max}) - f(u^-_{\mathrm{lim}})}{u_{\max} - u^-_{\mathrm{lim}}} = g'(u_{\max}),
    \]
    and the solution takes the form
    \begin{equation}\label{eq:case2C}
    (u, \theta)(t,x) =
    \begin{cases}
    (u^-, 1) & \text{if } x/t < \lambda, \\
    (u_{\max}, 1)& \text{if } \lambda < x/t < g^\prime(u_{\max}), \\
    (w, 0) & \text{if } x/t=g'(w), \ g^\prime(u_{\max}) < x/t < g^\prime(u^+), \\
    (u^+, 0) & \text{if } x/t > g^\prime(u^+),
    \end{cases}
    \end{equation}
    where the interface lies on the curve $x=g'(u_{\max})\,t$.
    \begin{remark}
    Case~2C is specific to the traffic setting, where the unknown $u$ represents the density of the vehicles and is therefore nonnegative ($u \ge 0$), while $f$ and $g$ are concave fluxes in a bounded interval $[0,u_{\max}]$. In this framework, the point $u^*$ in~\eqref{eq:tangent-point} may lie outside the admissible domain $[0,u_{\max}]$. This issue does not arise when $f$ and $g$ are defined on the whole line $\mathbb{R}$ as in~\cite{ABS2025}.
    \end{remark}
    \end{enumerate}
    We show below that the above Riemann solver and, in particular, the location of the interface, is consistent with the definition of weak solution \eqref{def:weak_sol}. 
    
    \begin{lemma}\label{lemma:sol2C}
The Riemann solver \eqref{eq:case2C} of Case 2C is a weak solution to \eqref{eq:original} with initial data 
\[
u(0, x)=
\begin{cases}
    u^{-} > u^-_{\mathrm{lim}}, &\quad x < 0,\\
    u^+, &\quad x > 0
\end{cases}
\]
in the sense of Definition \ref{def:weak_sol}.
\end{lemma}

\begin{proof}
Recall that $\lambda < g'(u_{\max})$. We want to show that for any test function $\varphi(t, x) \in C^\infty_c((0, T) \times \mathbb{R})$,
\begin{align*}
    &\int_0^{T} \int_{-\infty}^{\lambda t} \big( u^{-} \partial_t \varphi + f(u^-) \partial_x \varphi \big) \, dx \, dt + \int_0^{T}\int_{\lambda t}^{g'(u_{\max})t} \big( u_{\max} \partial_t \varphi + f(u_{\max}) \partial_x \varphi \big) \, dx \, dt \\
    &+\int_0^{T}\int_{g'(u_{\max})t}^{g'(u^+)t} \Big( (g')^{-1}(x/t)\partial_t \varphi + g\big((g')^{-1}(x/t)\big) \partial_x \varphi \Big) \, dx \, dt \\
    &+ \int_0^{T} \int_{g' (u^+) t}^{+\infty} \big( u^+\partial_t \varphi + g(u^+) \partial_x \varphi \big) \, dx \, dt =: \mathcal{I}_1 + \mathcal{I}_2 + \mathcal{I}_3 + \mathcal{I}_4 = 0.
\end{align*}

Integrating by parts, the first integral gives:
\begin{align*}
    \mathcal{I}_1 = \int_0^T \big( -\lambda u^- + f(u^-) \big) \varphi(t, \lambda t) \, dt.
\end{align*}
The second one yields:
\begin{align*}
    \mathcal{I}_2 = \int_0^T \Big[ \big( \lambda u_{\max} - f(u_{\max}) \big) \varphi(t, \lambda t) + \big( -u_{\max} g'(u_{\max}) + f(u_{\max}) \big) \varphi(t, g'(u_{\max})t) \Big] \, dt.
\end{align*}
By definition of the shock speed $\lambda = \frac{f(u^-) - f(u_{\max})}{u^- - u_{\max}}$, the jump terms at $x = \lambda t$ cancel out, so that:
\[
\mathcal{I}_1 + \mathcal{I}_2 = \int_0^T \big( - u_{\max} g'(u_{\max}) + \underbrace{f(u_{\max})}_{= \, g(u_{\max})} \big) \varphi(t, g'(u_{\max})t) \, dt.
\]
Next, for $\mathcal{I}_3 = \mathcal{I}_3^t + \mathcal{I}_3^x$, integration by parts with respect to $t$ and $x$ leads to:
\begin{align*}
    \mathcal{I}_3^t &= \int_0^T \Big[ - g'(u^+) u^+ \varphi(t, g'(u^+)t) + g'(u_{\max}) u_{\max}\varphi(t, g'(u_{\max})t) \Big] \, dt \\
    &\quad + \int_0^{T}\int_{g'(u_{\max})t}^{g'(u^+)t} \frac{x}{t^2} \frac{\varphi(t, x)}{g''((g')^{-1}(x/t))} \, dx \, dt,
\end{align*}
and
\begin{align*}
    \mathcal{I}_3^x &= \int_0^T \Big[ g(u^+) \varphi(t, g'(u^+)t) - \underbrace{g(u_{\max})}_{= \, f(u_{\max})} \varphi(t, g'(u_{\max})t) \Big] \, dt \\
    &\quad - \int_0^{T}\int_{g'(u_{\max})t}^{g'(u^+)t} \frac{x}{t^2} \frac{\varphi(t, x)}{g''((g')^{-1}(x/t))} \, dx \, dt.
\end{align*}
Finally, the boundary terms at $x = g'(u^+)t$ cancel with:
\begin{align*}
    \mathcal{I}_4 = \int_0^T \big( g'(u^+) u^+ - g(u^+) \big) \varphi(t, g'(u^+)t) \, dt.
\end{align*}
Summing $\mathcal{I}_1 + \mathcal{I}_2 + \mathcal{I}_3^t + \mathcal{I}_3^x + \mathcal{I}_4$, all boundary trace terms and domain integrals cancel out, proving that $\mathcal{I}_1 + \mathcal{I}_2 + \mathcal{I}_3 + \mathcal{I}_4 = 0$. Note that the cancellation between $\mathcal{I}_2$ and $\mathcal{I}_3^x$ relies on the continuity identity $f(u_{\max}) = g(u_{\max})$.
\end{proof}

\begin{remark}[On the location of the interface]
We may ask whether the interface could be located at $x = \tilde{\lambda}t$ with $\tilde{\lambda} < g'(u_{\max})$, rather than at $x = g'(u_{\max})t$. In this case, the candidate solution would take the form
\[
(u,\theta)(t,x)=
\begin{cases}
(u^-,1), & x/t < \lambda,\\
(u_{\max},1), & \lambda < x/t < \tilde{\lambda},\\
(w,0), & x/t = g'(w),\quad \tilde{\lambda} < x/t < g'(u^+),\\
(u^+,0), & x/t > g'(u^+).
\end{cases}
\]
However, the rarefaction wave associated with the flux $g$ would then originate from the characteristic speed $\tilde{\lambda}$. Hence, there must exist a state $\bar{w}$ such that $g'(\bar{w}) = \tilde{\lambda}$.
Since $\tilde{\lambda} < g'(u_{\max})$ and $g'$ is strictly decreasing due to the concavity of $g$, we necessarily have
\[
\bar{w} > u_{\max}.
\]
Therefore, the rarefaction wave would involve states strictly greater than the admissible maximal density $u_{\max}$. This violates the physical constraint $u \leq u_{\max}$. Consequently, no interface velocity $\tilde{\lambda} < g'(u_{\max})$ is admissible, leaving $\tilde{\lambda} = g'(u_{\max})$ as the unique choice.
\end{remark}

\begin{figure}[h!]
\centering
\begin{tikzpicture}[scale=0.8, baseline={(0,0)}]
  \draw[->] (-0.5,0) -- (5,0) node[right] {};
  \draw[->] (0,-0.5) -- (0,3.5) node[above] {};
  \draw[domain=0:4, smooth, variable=\x, red, thick] 
       plot ({\x}, {0.8*\x*(4-\x)});
  \draw[domain=0:4, smooth, variable=\x, blue, thick] 
       plot ({\x}, {0.5*\x*(4-\x)}); 
  \node at (2.8,2) [above] {\color{red}$f$};
  \node at (3.5,0.3) [above] {\color{blue}$g$};
  \node at (-0.2,0) [below] {$0$};
  \node at (4.2,0) [below] {$u_{\max}$};
  \draw[thick, black] (0.5,{0.8*0.5*3.5}) -- (2,{0.48*4.5});
  \draw[thick, black, dashed] (0.5,0) -- (0.5,{0.8*0.5*3.5});
  \node at (0.5,0) [below] {$u^-$};
  
  \draw[thick, black, dashed] (2.8,0) -- (2.8,{0.5*2.8*(4-2.8)});
  \node at (3,0) [below] {$u^+$};
  
  \draw[thick, black, dashed] (1.5,0) -- (1.5,{0.5*1.5*(4-1.5)});
  \node at (1.5,0) [below] {$u^*$};
  
  \draw[thick, black] (0.5,{0.8*0.5*3.5}) -- (2.8,{0.5*2.8*(4-2.8)});
\end{tikzpicture}
\begin{tikzpicture}[scale=0.8, baseline={(0,0)}]
  \draw[->] (-0.5,0) -- (5,0) node[right] {};
  \draw[->] (0,-0.5) -- (0,3.5) node[above] {};
  \draw[domain=0:4, smooth, variable=\x, red, thick] 
       plot ({\x}, {0.8*\x*(4-\x)});
  \draw[domain=0:4, smooth, variable=\x, blue, thick] 
       plot ({\x}, {0.5*\x*(4-\x)}); 
  \node at (1.2,2.8) [above] {\color{red}$f$};
  \node at (0.5,0.3) [above] {\color{blue}$g$};
  \node at (-0.2,0) [below] {$0$};
  \node at (4.2,0) [below] {$u_{\max}$};
  \draw[thick, black] (1.8,{0.8*1.8*(4-1.8)}) -- (3.5,{0.5*3.5*(4-3.5)});
  \draw[thick, black, dashed] (1.8,0) -- (1.8,{0.8*1.8*(4-1.8)});
  \node at (1.8,0) [below] {$u^-$};
  
  \draw[thick, black, dashed] (1,0) -- (1,{0.5*1*(4-1)});
  \node at (1,0) [below] {$u^+$};

  \draw[thick, black, dashed] (3.3,0) -- (3.3,{0.5*3.3*(4-3.3)});
  \node at (3.3,0) [below] {$u^*$};
  
  \draw[thick, black] (2.5,{0.8*2.5*(4-2.5)}) -- (4,0);
  \draw[thick, black, dashed] (2.5,0) -- (2.5,{0.8*2.5*(4-2.5)});
  \node at (2.5,0) [below] {$u^-_{\mathrm{lim}}$};
  \fill[black] (4,0) circle (2pt);
  \fill[black] (2.5,{0.8*2.5*(4-2.5)}) circle (2pt);
\end{tikzpicture}
\begin{tikzpicture}[scale=0.8, baseline={(0,0)}]
  \draw[->] (-0.5,0) -- (5,0) node[right] {};
  \draw[->] (0,-0.5) -- (0,3.5) node[above] {};
  \draw[domain=0:4, smooth, variable=\x, red, thick] 
       plot ({\x}, {0.8*\x*(4-\x)});
  \draw[domain=0:4, smooth, variable=\x, blue, thick] 
       plot ({\x}, {0.5*\x*(4-\x)}); 
  \node at (1.2,2.8) [above] {\color{red}$f$};
  \node at (0.5,0.3) [above] {\color{blue}$g$};
  \node at (-0.2,0) [below] {$0$};
  \node at (4.2,0) [below] {$u_{\max}$};
  \draw[thick, black] (2.5,{0.8*2.5*(4-2.5)}) -- (4,0);
  \draw[thick, black, dashed] (2.5,0) -- (2.5,{0.8*2.5*(4-2.5)});
  \node at (2.5,0) [below] {$u^-_{\mathrm{lim}}$};
  \fill[black] (4,0) circle (2pt);
  \fill[black] (2.5,{0.8*2.5*(4-2.5)}) circle (2pt);
  \draw[thick, black, dashed] (1.3,0) -- (1.3,{0.5*1.3*(4-1.3)});
  \node at (1.3,0) [below] {$u^+$};
  \draw[thick, black] (3.2,{0.8*3.2*(4-3.2)}) -- (4,0);
  \draw[thick, black, dashed] (3.2,0) -- (3.2,{0.8*3.2*(4-3.2)});
  \node at (3.2,0) [below] {$u^-$};
\end{tikzpicture}
\caption{Representation of the Riemann solution states in Case 2. Left: Case 2A; center: Case 2B; right: Case 2C.}
\label{fig:RScase2}
\end{figure}
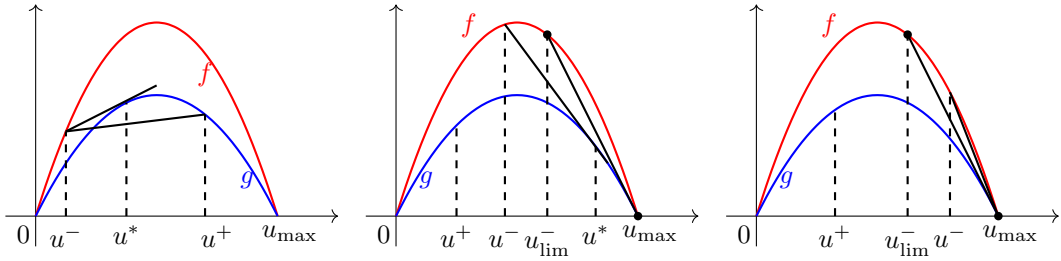
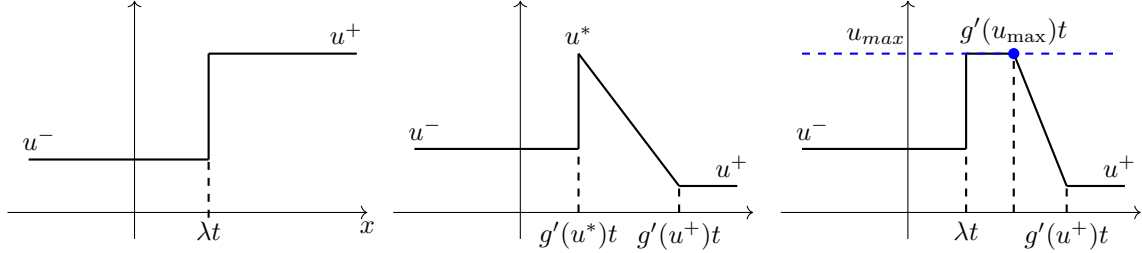
\begin{figure}[h!]
\centering
\begin{tikzpicture}[scale=0.7, baseline={(0,0)}]
  \draw[->] (-3.4,0) -- (3.4,0) node[below] {$x$};
  \draw[->] (-1,-0.5) -- (-1,4) node[above] {};

  \draw[thick, black] (-3,1) -- (0.4,1);
  \draw[thick, black] (0.4,3) -- (3.2,3);
  \draw[thick, black] (0.4,1) -- (0.4,3);
  
  \draw[thick, black, dashed] (0.4,-0.1) -- (0.4,1);
  \node at (0.4,0) [below] {$\lambda t$};
  \node at (-2.8,1) [above] {$u^-$};
  \node at (3,3) [above] {$u^+$};
  
\end{tikzpicture}
\begin{tikzpicture}[scale=0.7, baseline={(0,0)}]
  \draw[->] (-3.4,0) -- (3.4,0) node[right] {};
  \draw[->] (-1,-0.5) -- (-1,4) node[above] {};
  
  \draw[thick, black] (-3,1.2) -- (0.1,1.2);
  \draw[thick, black] (0.1,1.2) -- (0.1,3);
  \draw[thick, black] (0.1,3) -- (2,0.5);
  \draw[thick, black] (2,0.5) -- (3.1,0.5);
  
  \node at (0.1,0) [below] {$g'(u^*) t$};
  \draw[thick, black, dashed] (0.1,0) -- (0.1,1.2);
  
  \node at (2,0) [below] {$g^\prime(u^+)t$};
  \draw[thick, black, dashed] (2,0) -- (2,0.5);
  \node at (-2.8,1.2) [above] {$u^-$};
  \node at (3,0.5) [above] {$u^+$};
  \node at (0.1,3) [above] {$u^*$};
\end{tikzpicture}
\begin{tikzpicture}[scale=0.7, baseline={(0,0)}]
  \draw[->] (-3.4,0) -- (3.4,0) node[right] {};
  \draw[->] (-1,-0.5) -- (-1,4) node[above] {};
  
  \draw[thick, black] (-3,1.2) -- (0.1,1.2);
  \draw[thick, black] (0.1,1.2) -- (0.1,3);
  \draw[thick, black] (0.1,3) -- (1,3);
  \draw[thick, black] (1,3) -- (2,0.5);
  \draw[thick, black] (2,0.5) -- (3.1,0.5);
  
  \node at (0.1,0) [below] {$\lambda t$};
  \draw[thick, black, dashed] (0.1,0) -- (0.1,1.2);
  
  \node at (2,0) [below] {$g^\prime(u^+)t$};
  \draw[thick, black, dashed] (2,0) -- (2,0.5);
  \node at (-2.8,1.2) [above] {$u^-$};
  \node at (3,0.5) [above] {$u^+$};
  \fill[blue] (1,3) circle (3pt);
  \node at (1,3) [above] {$g'(u_{\max}) t$};
  \draw[thick, black, dashed] (1,0) -- (1,3);
  \draw[thick, blue, dashed] (-3,3) -- (3,3);
  \node at (-1.6,3)[above]  {$u_{max}$};
\end{tikzpicture}
\caption{Solution of the Riemann problem in Cases 2A, 2B, and 2C, respectively.}\label{fig:solCase2}
\end{figure}


\subsection{Case 3: $\theta^-=0$, $\theta^+=1$ ($g\rightarrow f$).}\label{sec:case3}
Recall that $v^*$ in \eqref{eq:vstar} is the unique tangent point satisfying $v^*<u^+$, namely 
    \begin{equation*}
    g^\prime(v^*) = \frac{f(u^+)-g(v^*)}{u^+-v^*}, \quad v^*<u^+.
    \end{equation*}
    Under the assumption $f'(0)\neq g'(0)$, $v^*$ exists if and only if $u^+>u^+_\mathrm{lim}$ where $u^+_\mathrm{lim}$ is given in \eqref{eq:upiulim}.
    We then distinguish the following two cases: $u^+\leq u^+_\mathrm{lim}$ and $u^+>u^+_\mathrm{lim}$.

\begin{enumerate}
    \item[(A)] If $u^+>u^+_{\mathrm{lim}}$, $u^- \leq v^* < u^+$, then the solution is a single jump from $u^-$ to $u^+$ with speed
    $$
    \lambda = \displaystyle\frac{g(u^-)-f(u^+)}{u^--u^+}, \quad u^--u^+ < 0.
    $$
    By Lemma \ref{lemma:up_jump}, this jump is Lax admissible. In this case the location of the jump is an interface, and we have
    \[
        (u,\theta)(t,x) = \left\{\begin{array}{cc}
              (u^-,0) &  x/t < \lambda \\
              (u^+,1) &  x/t > \lambda.
        \end{array}\right. 
    \]
    
    \item[(B)] If $u^+>u^+_{\mathrm{lim}}$, $v^* < u^+$ and $u^- \geq v^*$, then we have a rarefaction fan on the flux $g$ between $g^\prime(u^-)$ and $g^\prime(v^*)$ ($g^\prime(u^-) < g^\prime(v^*)$), and a shock that jumps from $v^*$ to $u^+$ with speed
    $$
    \lambda = g^\prime(v^*),
    $$
    which is admissible by Lemma \ref{lemma:up_jump}. In summary, the solution consists of a rarefaction followed by a shock:
    \begin{equation}\label{eq:sol3b}
     (u,\theta)(t,x) = \left\{\begin{array}{cc}
                 (u^-,1)     & x/t < g^\prime(u^-) \\
                 (w,1) &  x/t=g'(w), \ g^\prime(u^-) < x/t< \lambda\\
                 (u^+,0)  &  x/t > \lambda.
            \end{array}
            \right.   
    \end{equation}
    In this case, the jump is also an interface.
    
    \item[(C)] If $u^+ \leq u^+_{\mathrm{lim}}$, the value $v^*$ is not defined. In such a case, i.e. if $u^+<u^+_{\mathrm{lim}}$ we have a rarefaction fan on $g$ between $g^\prime(u^-) < g^\prime(0)$ and a (positive) shock that jumps from the vacuum $0$ to $u^+$. Similarly to the previous Case 2C, we define the speed of the jump as
    $$
    	\lambda = \frac{f(u^+)-f(0)}{u^+-0} = \frac{f(u^+)}{u^+},
    $$
    which is Lax-admissible. 
    The same argument used in Case 2C to determine the interface speed under the constraint $u \leq u_{\max}$ (see Lemma~\ref{lemma:sol2C}) applies to the constraint $u \geq 0$. Therefore, the only admissible choice is for the interface to propagate
    along the curve $x=g'(0)t$. By the properties of $f$, the jump speed verifies $\lambda>g'(0)$ and the solution takes the form
   \[
            (u,\theta)(t,x) = \left\{\begin{array}{cc}
                 (u^-,0)     & x/t < g^\prime(u^-) \\
                 (w,0)       & x/t=g'(w),\ g^\prime(u^-) < x/t< g'(0) \\
                 (0,1)       & g'(0)< x/t < \lambda\\
                 (u^+,1)     &  x/t > \lambda.
            \end{array}
            \right.
    \]
 We observe that, under these specific conditions, a 'vacuum state' (or zero-density zone) can emerge within the space. This aligns with vehicle traffic dynamics: since the rarefaction between $u^-$ and $0$ takes place along the lower flow $g$, the traveling vehicles — despite traveling at their maximum free-flow velocity — cannot bridge the gap to reach the faster-moving vehicles ahead, which are governed by the higher flow $f$.
\end{enumerate}
\begin{figure}[h!]
\centering
\begin{tikzpicture}[scale=0.7]
  \draw[->] (-3.4,0) -- (3.4,0) node[right] {$x$};
  \draw[->] (-1,-0.5) -- (-1,4) node[left] {$u$};

  \draw[thick, black] (-3,1) -- (0.4,1);
  \draw[thick, black] (0.4,3) -- (3.2,3);
  \draw[thick, black] (0.4,1) -- (0.4,3);
  
  \draw[thick, black, dashed] (0.4,-0.1) -- (0.4,1);
  \node at (0.4,0) [below] {$\lambda t$};
  \node at (-2.8,1) [above] {$u^-$};
  \node at (3,3) [above] {$u^+$};
\end{tikzpicture}
\begin{tikzpicture}[scale=0.7]
  \draw[->] (-3.4,0) -- (3.4,0) node[right] {};
  \draw[->] (-1,-0.5) -- (-1,4) node[left] {$u$};
  
  \draw[thick, black] (-3,3) -- (0.1,3);
  \draw[thick, black] (0.1,3) -- (2,0.5);
  \draw[thick, black] (2,0.5) -- (2,2);
  \draw[thick, black] (2,2) -- (3.1,2);
  
  \node at (0.1,0) [below] {$g^\prime(u^-) t$};
  \draw[thick, black, dashed] (0.1,0) -- (0.1,3);
  
  \node at (2,0) [below] {$\lambda t$};
  \draw[thick, black, dashed] (2,0) -- (2,0.5);
  \node at (-2.8,3) [above] {$u^-$};
  \node at (3,2) [above] {$u^+$};
  \node at (2.3,0.5)  {$v^*$};
  \end{tikzpicture}
  \begin{tikzpicture}[scale=0.7]
  \draw[->] (-3.4,0) -- (3.4,0) node[right] {};
  \draw[->] (-1,-0.5) -- (-1,4) node[left] {$u$};
  
  \draw[thick, black] (-3,3) -- (0.1,3);
  \draw[thick, black] (0.1,3) -- (1.4,0);
  \draw[line width=2pt, blue] (1.4,0) -- (2,0);
  \draw[thick, black] (2,0) -- (2,2);
  \draw[thick, black] (2,2) -- (3.1,2);
  \node at (1.3,0) [below] {$g^\prime(0) t$};
  
  
  \draw[thick, black, dashed] (2,0) -- (2,0.5);
  \node at (-2.8,3) [below] {$u^-$};
  \node at (3,2) [above] {$u^+$};
  \node at (-1.2,0)[below]  {$0$};
  \fill[blue] (1.4,0) circle (3pt);
\end{tikzpicture}
\caption{Solution to the Riemann problem in Case 3A, 3B, and 3C  respectively.}\label{fig:solCase3}
\end{figure}
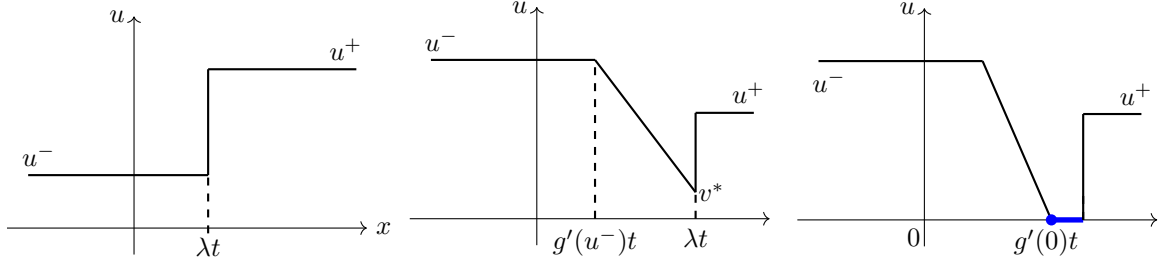

\subsection{Case 4: $\theta^-=1$, $\theta^+=1$ ($f\rightarrow f$).}\label{sec:case4}
Let $u^*>u^-$, $v^* < u^+$ be defined as in Lemma \ref{lemma:up_jump}, and $u^-_{\mathrm{lim}}$, $u^+_{\mathrm{lim}}$ be given by \eqref{eq:umenolim} and \eqref{eq:upiulim} respectively. 
\begin{enumerate}
	\item[(A)] $u^-\leq u^+$. In this case $u=u(t,x)$ is the solution to the scalar conservation law
	\begin{equation}\label{eq:CLf}
	u_t+f(u)_x = 0
	\end{equation}
	with initial data \eqref{eq:rp} and $\theta(t,x)\equiv 1$ for all $t,x$. Indeed, for $u^-\leq u^+$ the solution of \eqref{eq:CLf} is a single upward jump. Hence, $u_x\geq 0$ at a.e. point $(t,x)$ and $\theta(t,x)\equiv 1$ for all $(t,x)$ satisfies Definition \ref{def:weak_sol}.
	
	\item[(B)] $u^-> u^+$. In this case, the solution to \eqref{eq:CLf} involves a rarefaction; however, it does not solve our problem (see Remark~\ref{remark:rar}). The correct solution instead consists of two jumps: first from $f$ to $g$, connecting $u^-$ to $u^*$, and then from $g$ to $f$, connecting $v^*$ to $u^+$, with a rarefaction fan for $g$ between $u^*$ and $v^*$. Referring to Cases~2B and~3B, when $u^*$ and $v^*$ are well defined (i.e., $u^- < u^-_{\mathrm{lim}}$ and $u^+ > u^+_{\mathrm{lim}}$), we obtain the following solution:
    \begin{equation}\label{eq:sol4b}
        (u,\theta)(t,x) = \left\{\begin{array}{cc}
             (u^-,1)     & x/t < g^\prime(u^*), \\
             (w,0) &  x/t=g'(w), \ g^\prime(u^*) < x/t< g^\prime(v^*),\\
             (u^+,1)  &  x/t > g^\prime(v^*).
        \end{array}
        \right.   
    \end{equation}
    
    \item[(C)] When $u^*$ and/or $v^*$ are not defined ($u^->u^-_{\mathrm{lim}}$ and/or $u^+<u^+_{\mathrm{lim}}$), the solution is described in Case 2C and 3C respectively, see Figure \ref{fig:RScase4c}.
\end{enumerate}
\begin{figure}[h!]
\centering

\begin{tikzpicture}[scale=1.0]
  \draw[->] (-0.5,0) -- (5,0) node[right] {};
  \draw[->] (0,-0.5) -- (0,3.5) node[above] {};
  \draw[domain=0:4, smooth, variable=\x, red, thick] 
       plot ({\x}, {0.8*\x*(4-\x)});
  \draw[domain=0:4, smooth, variable=\x, blue, thick] 
       plot ({\x}, {0.5*\x*(4-\x)}); 
  \node at (-0.2,0) [below] {$0$};
  \node at (4.2,0) [below] {$u_{\max}$};
  
  \draw[thick, black, dashed] (1.5,0) -- (1.5,{0.8*1.5*(4-1.5)});
  \node at (1.5,0) [below] {$u^+$};
  \draw[thick, black] (1.5,{0.8*1.5*(4-1.5)}) -- (0.3,{0.5*0.3*(4-0.3)});
  \draw[thick, black, dashed] (0.4,0) -- (0.4,{0.5*0.4*(4-0.4)});
  \node at (0.4,0) [below] {$v^*$};
  
  \draw[thick, black, dashed] (2.4,0) -- (2.4,{0.8*2.4*(4-2.4)});
  \node at (2.4,0) [below] {$u^-$};
  \draw[thick, black] (2.4,{0.8*2.4*(4-2.4)}) -- (3.4,{0.5*3.4*(4-3.4)});
  \draw[thick, black, dashed] (3.3,0) -- (3.3,{0.5*3.3*(4-3.3)});
  \node at (3.3,0) [below] {$u^*$};
  
\end{tikzpicture}
\begin{tikzpicture}[scale=1.0]
  \draw[->] (-3.4,0) -- (3.4,0) node[right] {};
  \draw[->] (-1,-0.5) -- (-1,4) node[left] {$u$};
  
  \draw[thick, black] (-3,2) -- (0.1,2);
  \draw[thick, black] (0.1,2) -- (0.1,3);
  \draw[thick, black] (0.1,3) -- (1.5,0.5);
  \draw[thick, black] (1.5,0.5) -- (1.5,1.5);
  \draw[thick, black] (1.5,1.5) -- (3,1.5);
  
  \node at (0.1,0) [below] {$g^\prime(u^*) t$};
  \draw[thick, black, dashed] (0.1,0) -- (0.1,3);
  
  \node at (1.5,0) [below] {$g^\prime(v^*) t$};
  \draw[thick, black, dashed] (1.5,0) -- (1.5,0.5);
  \node at (-2.8,2) [above] {$u^-$};
  \node at (3,1.5) [above] {$u^+$};
  \node at (1.8,0.5)  {$v^*$};
  \node at (0.1,3) [above] {$u^*$};
\end{tikzpicture}
\caption{Solution to the Riemann problem in Case 4B.}\label{fig:RScase4}
\end{figure}
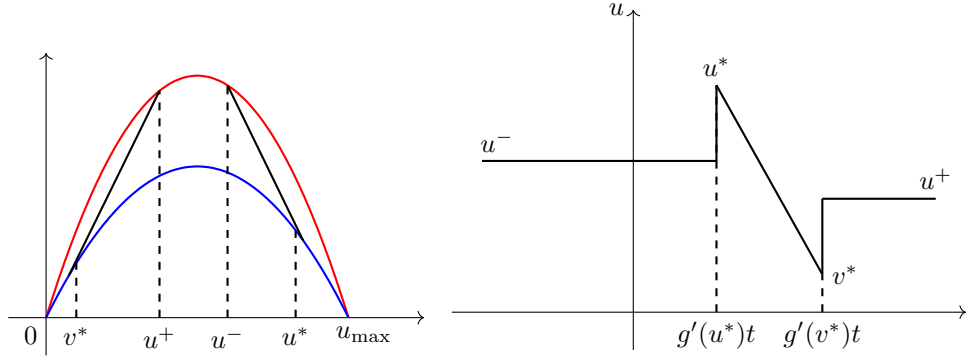
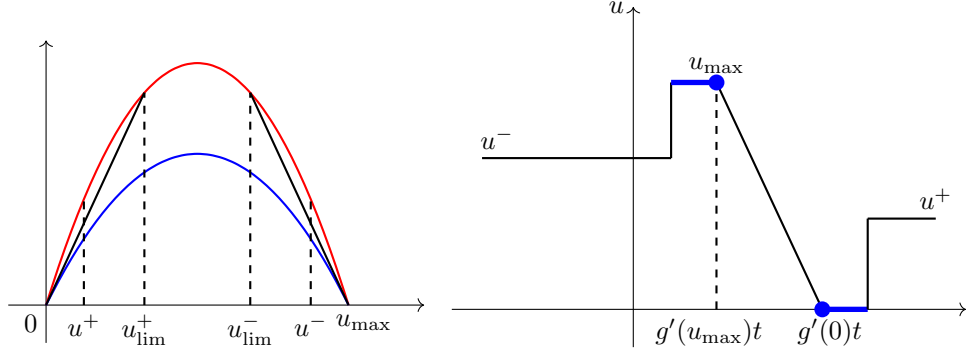
\begin{figure}[h!]
\centering

\begin{tikzpicture}[scale=1.0]
  \draw[->] (-0.5,0) -- (5,0) node[right] {};
  \draw[->] (0,-0.5) -- (0,3.5) node[above] {};
  \draw[domain=0:4, smooth, variable=\x, red, thick] 
       plot ({\x}, {0.8*\x*(4-\x)});
  \draw[domain=0:4, smooth, variable=\x, blue, thick] 
       plot ({\x}, {0.5*\x*(4-\x)}); 
  \node at (-0.2,0) [below] {$0$};
  \node at (4.2,0) [below] {$u_{\max}$};
  
  \draw[thick, black, dashed] (0.5,0) -- (0.5,{0.8*0.5*(4-0.5)});
  \node at (0.5,0) [below] {$u^+$};
  \draw[thick, black] (0,0) -- (1.3,{0.8*1.3*(4-1.3)});
  \draw[thick, black, dashed] (1.3,0) -- (1.3,{0.8*1.3*(4-1.3)});
  \node at (1.3,0) [below] {$u^+_{\mathrm{lim}}$};
  
  \draw[thick, black, dashed] (3.5,0) -- (3.5,{0.8*3.5*(4-3.5)});
  \node at (3.5,0) [below] {$u^-$};
  \draw[thick, black] (4,{0.8*4*(4-4)}) -- (2.7,{0.8*2.7*(4-2.7)});
  \draw[thick, black, dashed] (2.7,0) -- (2.7,{0.8*2.7*(4-2.7)});
  \node at (2.7,0) [below] {$u^-_{\mathrm{lim}}$};
  
\end{tikzpicture}
\begin{tikzpicture}[scale=1.0]
  \draw[->] (-3.4,0) -- (3.4,0) node[right] {};
  \draw[->] (-1,-0.5) -- (-1,4) node[left] {$u$};
  
  \draw[thick, black] (-3,2) -- (-0.5,2);
  \draw[thick, black] (-0.5,2) -- (-0.5,3);
  \draw[line width=2pt, blue] (-0.5,3) -- (0.1,3);
  \draw[thick, black] (0.1,3) -- (1.5,0);
  \draw[line width=2pt, blue] (1.5,0) -- (2.1,0);
  \draw[thick, black] (2.1,0) -- (2.1,1.2);
  \draw[thick, black] (2.1,1.2) -- (3,1.2);
  
  \node at (0,0) [below] {$g^\prime(u_{\mathrm{max}}) t$};
  \draw[thick, black, dashed] (0.1,0) -- (0.1,3);
  \fill[blue] (0.1,3) circle (3pt);
  
  \node at (1.6,0) [below] {$g^\prime(0) t$};
   \fill[blue] (1.5,0) circle (3pt);
  \node at (-2.8,2) [above] {$u^-$};
  \node at (3,1.2) [above] {$u^+$};
  \node at (0.1,3) [above] {$u_{\mathrm{max}}$};
\end{tikzpicture}
\caption{Solution to the Riemann problem in Case 4C.}\label{fig:RScase4c}
\end{figure}
\section{Existence of weak solutions}
The aim of this section is to establish an existence result for the Cauchy problem associated with \eqref{eq:original}, for initial data satisfying (H1)--(H2). We follow the general strategy developed in \cite{ABS2025}, while introducing different Riemann solvers to account for the fact that, in the present setting, the admissible domain $[0,u_{\mathrm{max}}]$ is bounded.

\begin{theorem}
There exists a weak solution $(u(t,x),\theta(t,x))$ to the Cauchy problem associated with \eqref{eq:original}, in the sense of Definition~\ref{def:weak_sol}, for any initial data satisfying (H1)--(H2).
\end{theorem}
\begin{remark}[Minimization of the number of interfaces is a selection principle]
    As in \cite{ABS2025}, uniqueness does not hold. However, minimizing the number of interfaces, we can select a unique solution.
\end{remark}
The proof is divided in several steps. The heart of the strategy is to determine and study an ODE for the (moving) interface that separates the region where the flux is $f(u)$ from the region where the flux is $g(u)$.

We introduce the initial data. Let $\bar u(x)$ be monotone separately on the left and on the right. The initial data read:
\begin{equation}\label{eq:initialdata}
    \begin{cases}
        u(0, x)= \bar u(x), \\
        \theta (0, x)=\bar \theta (x),
    \end{cases}
    \quad \lim_{x \rightarrow 0\pm} \bar u(x)=u^{\pm}, \quad \bar \theta (x)=\begin{cases}
        \theta^-, \quad x<0,\\
        \theta^+, \quad x>0,
    \end{cases} \quad \theta^{\pm} \in \{0, 1\}.
\end{equation}
We define the solution $u(t,x)$ through a moving interface $y(t)$, whose evolution is determined by the equation 
\begin{equation}\label{eq:Yode}
\dot y(t) = H(t,y(t)), \qquad y(0)=0 ,
\end{equation}
where the function $H(t,x)$ is built through
algebraic operations and compositions involving the traces of 
functions which are only 
$L^1_{\mathrm{loc}}(\R)$ with bounded variation.
In particular, the resolution of \eqref{eq:Yode} lies beyond the scope
of the classical Cauchy--Lipschitz theory for ODEs.
To deal with this low regularity, following \cite{ABS2025}, we appeal to an existence and uniqueness result for ODEs based on the notion of \emph{locally directional bounded variation} \cite{bressan1988}, which we introduce below. 

\begin{definition}\label{def:localBV}
A vector field $H : \mathbb{R}^2 \to \mathbb{R}$ is said to have 
\emph{locally bounded directional variation} if, for every 
$(t_0,x_0) \in \mathbb{R}^2$, there exist $\delta>0$ and a constant $C>0$
such that
\[
\sum_{i=1}^N 
\bigl| H(t_i,x_i) - H(t_{i-1},x_{i-1}) \bigr|
\;\le\; C
\]
for every finite sequence $(t_i,x_i)$, $i=0,\ldots,N$, satisfying
\[
(t_0,x_0) \prec (t_1,x_1) \prec \cdots \prec (t_N,x_N),
\qquad t_N < t_0 + \delta .
\]
Here $\prec$ denotes the partial order on $\mathbb{R}^2$ induced by the cone
\begin{equation}\label{eq:gammaM}
    \Gamma_M := \{ (t,x) \in \mathbb{R}^2 : |x| \le M t\}, 
\end{equation}
namely
\begin{align}\label{eq:propx}
(t,x) \prec (t',x')
\quad \Longleftrightarrow \quad
|x - x'| \le M (t - t') .
\end{align}
\end{definition}

Observe that, if $H$ has locally bounded directional variation,
then~\eqref{eq:Yode} has a unique solution by \cite{bressan1988}. Specifically, we appeal to \cite{bressan1988} in \textbf{Case 2A}, while \textbf{Case 2B} does not fulfill the assumptions of the existence and uniqueness theorem for ODEs with discontinuous coefficients in \cite{bressan1988} and, therefore, the construction of that specific solution will involve a fixed point argument as in \cite[Case 2B]{ABS2025}. 
The same two strategies employed in \textbf{Case 2} are also used in \textbf{Case 3} and \textbf{Case 4}, as appropriate.

\medskip

We distinguish four different cases, according to Section \ref{sec:solvers}.

\subsection{Case 1. $\theta^-=0, \theta^+=0$ ($g \rightarrow g$)} 
In this case, we simply solve the conservation law $u_t+g(u)_x=0$ with initial data $(u(0, x), \theta(0, x)\equiv 0)$ in \eqref{eq:initialdata} such that $\bar u(x)$ is decreasing. We distinguish two cases.
\subsubsection{Case 1A: $u^-\ge u^+$.} 
The solution exhibits only rarefaction waves, with a centered rarefaction fan at $x=0$ between $u^-$ and $u^+$.
Since the solution is decreasing, we can set $\theta(t, x)=0$ for all $(t, x)$ and there is no interface.
\subsubsection{Case 1B: $u^-<u^+$.} 
The solution exhibits rarefaction waves and an upward jump in $x=0$. However, when considering the flux $g(u)$, the solution is expected to be decreasing; therefore, we need to introduce an interface. 
We denote by $u^\flat(t, x), u^\sharp(t, x)$ the solutions to the conservation law 
\[u_t+g(u)_x=0\]
with initial data such that
\begin{equation}\label{eq:sys_ub_ud}
    u^\flat (0, x)=\begin{cases}
    \bar u(x), \quad x<0, \\
    u^-, \quad x>0
\end{cases}\qquad u^\sharp(0, x)=\begin{cases}
    u^+, \quad x<0, \\
    \bar u(x), \quad x>0,
\end{cases}
\end{equation}
where $\bar u(x)$ decreases in both $x<0$ and $x>0$. Gluing together the above solutions, we write:
\begin{align*}
    u(t, x)=\begin{cases}
        u^\flat(t, x), \quad x \le y (t), \\
        u^\sharp(t, x), \quad x > y (t),
    \end{cases}\quad 
    \theta(t, x)=\begin{cases}
        0, \quad x\neq y (t), \\
        1, \quad x=y (t), 
    \end{cases}
\end{align*}
where $y(t)$ is an evolving interface to be determined. 
According to the classical theory of conservation laws \cite{bressan2000},
any weak solution in the sense of Definition~\ref{def:weak_sol}
must satisfy the Rankine--Hugoniot condition across the interface.
Consequently, the interface position $y(t)$ is required to solve
\begin{equation*}
\dot y(t) = H(t,y(t)), \qquad y(0)=0 ,
\end{equation*}
where, in the current Case 1B,
\begin{align}\label{eq:H1b}
    H(t, y(t))=\frac{g(u^\flat(t, y(t))-g(u^\sharp(t, y(t))}{u^\flat(t, y(t))-u^\sharp(t, y(t))}.
\end{align}
This is a classical Rankine-Hugoniot shock and, therefore, the existence and uniqueness of solutions to \eqref{eq:Yode} in this case follows from the standard theory. 


\subsection{Case 2. $\theta^-=1, \theta^+=0$ ($f \rightarrow g$)}
Consider the initial data $(u(0, x), \theta (0, x))$ in \eqref{eq:initialdata}, where $\bar u(x)$ increases for $x<0$ and decreases for $x>0$.
We solve the two Cauchy problems below:
\begin{align}\label{eq:sys_ub}
u^\flat_t+f(u^\flat)_x=0, \quad
u^\flat(0,x)&=\begin{cases}
\bar u(x), \quad x<0,\\
u^-, \quad x>0, \qquad u^-=\lim_{x \rightarrow 0-} \bar u(x);
\end{cases}\\
\label{eq:sys_ud}
u^\sharp_t+g(u^\sharp)_x=0, \quad 
u^\sharp(0,x)&=\begin{cases}
u_{\max}, \quad x<0,\\
\bar u(x), \quad x>0.
\end{cases}
\end{align}
Since $\bar u$ is increasing for $x<0$ and decreasing for $x>0$, the solution $u^\sharp$ exhibits only rarefaction waves.
In contrast, $u^\flat$ develops shocks. The rarefaction fan centered at the origin can be implicitly defined as:
\begin{equation}\label{eq:sol_udiesis}
u^\sharp(t,x) =
\begin{cases}
w, &\quad x/t=g^\prime(w) \leq g^\prime(u^+),\\
\bar u(\xi) &\quad \mbox{ for } \xi > 0, \quad  x= \xi + g^\prime(\bar u(\xi)) t.
\end{cases}
\end{equation}
The solution $u=u(t,x)$ of the Cauchy problem associated with equation \eqref{eq:original} is obtained by gluing together the two solutions:
\begin{equation}\label{eq:sol_gluing}
u(t,x) = \left\{\begin{array}{ll}
u^\flat(t,x) & \mbox{ for } x < y(t),
\medskip\\
u^\sharp(t,x) & \mbox{ for } x > y(t),
\end{array}\right.
\end{equation}
for an interface $y=y(t)$ satisfying \eqref{eq:Yode}, where we impose the Rankine-Hugoniot condition:
\begin{equation}\label{eq:H}
H(t,x)
=
\frac{f\big(u^\flat(t,x-)\big)-g\big(u^\sharp(t,x+)\big)}
     {u^\flat(t,x-)-u^\sharp(t,x+)}.
\end{equation}
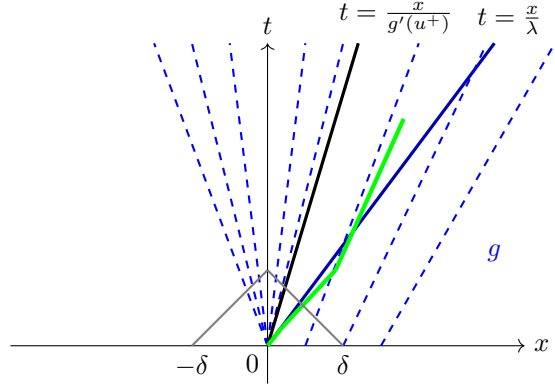
\begin{figure}[h!]
\centering
\begin{tikzpicture}[scale=1]
  \draw[->] (-3.4,0) -- (3.4,0) node[right] {$x$};
  \draw[->] (0,-0.5) -- (0,4) node[above] {$t$};
  
  \draw[thick, blue, dashed] (0,0) -- (-1.5,4);
  \draw[thick, blue, dashed] (0,0) -- (-1,4);
  \draw[thick, blue, dashed] (0,0) -- (-0.5,4);
  \draw[thick, blue, dashed] (0,0) -- (0.5,4);
  \draw[thick, blue, dashed] (0,0) -- (1,4);
  \draw[very thick, black] (0,0) -- (1.2,4);
  \node at (1.7,4) [above] {$t=\frac{x}{g^\prime(u^+)}$};
  \draw[very thick, blue!70!black] (0,0) -- (3,4);
  \node at (3.2,4) [above] {$t=\frac{x}{\lambda}$};
  \draw[thick, blue, dashed] (0.5,0) -- (2,4);
  \draw[thick, blue, dashed] (1,0) -- (2.9,4);
  \draw[thick, blue, dashed] (1.5,0) -- (3.8,4);
  \node at (3,1) [above] {\color{blue}$g$};
  \node at (-0.2,0) [below] {$0$};
  \draw[ultra thick, green] (0,0) -- (0.2,0.25);
  \draw[ultra thick, green] (0.2,0.25) -- (0.9,1);
  \draw[ultra thick, green] (0.9,1) -- (1.8,3);
  \draw[thick, gray] (-1,0) -- (0,1);
  \draw[thick, gray] (0,1) -- (1,0);
   \node at (-1,0) [below] {$-\delta$};
   \node at (1,0) [below] {$\delta$};
\end{tikzpicture}
\caption{Interface location in Case 2A (green line) along with the characteristic curves of solution $u^\sharp(t,x)$.}\label{fig:solCase2a}
\end{figure}

According to Section \ref{sec:case2}, we distinguish three sub-cases: \textbf{Case 2A}, \textbf{Case 2B}, \textbf{Case 2C}. In particular, we will consider initial data of the form \eqref{eq:sys_ub}, \eqref{eq:sys_ud} with different values of $\bar u(0+)$ in \eqref{eq:sys_ud} depending on the specific Case 2A and Case 2B, while in Case 2C the value $u^-$ in \eqref{eq:sys_ub} will be replaced by $u_{\mathrm{max}}$.

\subsubsection{Case 2A: $u^-<u^*<u^+$.}
The aim of this section is to prove the following.
\begin{proposition}\label{prop:main_2AI}
Let $f,g$ be as in Assumption \ref{assump}.
Let $u^\flat(t, x)$ and $u^\sharp(t, x)$ be the solutions to the Cauchy problems \eqref{eq:sys_ub}-\eqref{eq:sys_ud} with $\bar u(0+)=u^+$. Then, there exists a time $t_0>0$ and a constant $M>0$ such that, given any curve $x(t)$ satisfying \eqref{eq:propx}, denoting
\[
u^\flat(t)=u^\flat(t,x(t)), 
\qquad 
u^\sharp(t)=u^\sharp(t,x(t)) \quad \text{with} \quad x(t) \in \Gamma_M \quad \text{as} \quad \eqref{eq:gammaM},
\]
the following hold.
\begin{enumerate}
    \item The functions $u^\flat(t)$ and $u^\sharp(t)$ have bounded variation on $[0, t_0]$:
\begin{equation}\label{eq:TVbound_2AI}
    TV_{t \in [0, t_0]} (u^\flat(t))+TV_{t \in [0, t_0]} (u^\sharp(t)) \le C \cdot  TV(\bar u (x)),
\end{equation}
where $\bar u(x)$ is the initial data in \eqref{eq:sys_ub} and $\eqref{eq:sys_ud}$ and $C>0$ is a universal constant.
    \item There exists a constant $\mathfrak C=\mathfrak C (|u^+-u^-|)>0$ such that
    \[
|u^\flat(t)-u^\sharp(t)| \;\ge\; \mathfrak C
\qquad \text{for all } t\in[0,t_0].
\]
\item The function $H(t, x(t))$ in \eqref{eq:H}, associated to the Cauchy problem \eqref{eq:Yode}, has bounded variation on $[0,t_0]$.
\end{enumerate} 
\end{proposition}
We first prove that, for small times, the interface $y(t)$ in \eqref{eq:sol_gluing}, solving \eqref{eq:Yode}:
\begin{itemize}
\item 
satisfies the \emph{transversality property}
\begin{equation}
    \label{eq:lax-interface}
    g^\prime(u^\sharp(t, y(t)) < H(t, y(t)) < f^\prime(u^\flat(t, y(t));
\end{equation}
\item lies on the right of the centered rarefaction wave at the origin (see Figure \ref{fig:solCase2a}), i.e.
\[y(t) \ge g^\prime(u^+) t.\]
\end{itemize}
\begin{remark}
In Case~2A (see below), the transversality condition is met, allowing a direct application of \cite{bressan1988}. In contrast, it is not strictly satisfied in Case~2B, where we follow \cite{ABS2025,bressan1988} and construct an explicit solution to the ODE~\eqref{eq:Yode}.
\end{remark}
\begin{lemma}
Let $y(t)$ be the solution of the Cauchy problem \eqref{eq:Yode}. There exists $t_0>0$ such that, for all $t \in [0, t_0]$,
\[g^\prime(u^\sharp(t, y(t)) < H(t, y(t)) < f^\prime(u^\flat(t, y(t)).\]
Moreover,
\begin{equation}\label{eq:Ydomain}
y(t) \ge t g^\prime(u^+).
\end{equation}
\end{lemma}
\begin{proof}
Because $g$ is concave, for $u>u^+>u^-$ we have
\[
\frac{g(u^-)-g(u)}{u^- - u} \;>\; g^\prime(u).
\]
Since $f(u^-) \ge g(u^-)$, 
\[
\frac{f(u^-)-g(u)}{u^- - u} \;>\; g^\prime(u).
\]
Consider the region $x \le g^\prime(u^+)t$.  
From the explicit representation~\eqref{eq:sol_udiesis}, in this region
\[
g^\prime\bigl(u^\sharp(t,x)\bigr)=\frac{x}{t} \;\le\; g^\prime(u^+).
\]
By the monotonicity of $g^\prime$, this implies $u^\sharp(t,x)> u^+$, and therefore
\begin{align}\label{eq:lowerCase2A}
\frac{f(u^-)-g\bigl(u^\sharp(t,x)\bigr)}{u^- - u^\sharp(t,x)}
\;>\;
g^\prime\bigl(u^\sharp(t,x)\bigr)
=
\frac{x}{t}.
\end{align}

Since $u^\flat (t, x)<u^-$ and $f(x)$ is concave, there exists $\eta \in (0, 1)$ such that
\[u^- = \eta u^\flat+ (1-\eta) u^\sharp \; \Rightarrow \; f(u^-)>\eta f(u^\flat)+ (1-\eta) f(u^\sharp).\]
Subtracting $g(u^\sharp)$,
\[f(u^-)-g(u^\sharp)>\eta (f(u^\flat)-g(u^\sharp)) + (1-\eta) (f(u^\sharp)-g(u^\sharp))>\eta (f(u^\flat)-g(u^\sharp)).\]
Observing 
\[u^- - u^\sharp = \eta (u^\flat -u^\sharp)<0,\]
allows us to conclude that
\[\frac{f(u^-)-g(u^\sharp)}{u^--u^\sharp}<\frac{f(u^\flat)-g(u^\sharp)}{u^\flat -u^\sharp}.\]

\noindent Using \eqref{eq:lowerCase2A}, this gives
\begin{equation}\label{eq:onesidelax}
\dot y(t)=H(t,y(t))
\;>\;
g^\prime\!\left(u^\sharp(t,x)\right)
=
\frac{x}{t}.
\end{equation}
In particular, 
\[
\dot y(t) > g^\prime(u^+),
\]
which implies that the interface cannot enter the region $x<g^\prime(u^+)t$.
\medskip

Now, consider the region $x>y(t)\ge g^\prime(u^+)t$. By definition of $H(t, y(t))$, in this region, for $t>0$, 
\[\dot y(t)=H(t, y(t))>g'(u^+).\]
Taking $u^\sharp (t, x)$ sufficiently close to $u^+$ for $(t, x)$ near $(0, 0)$, we deduce
\[\dot y(t)=H(t, y(t))>g'(u^\sharp (t, x)).\]
Since $x> y (t) \ge g'(u^+) t$, by monotonicity we have 
$u^\sharp(t, x) \le u^\sharp (t, y(t))$ and 
\[g^\prime(u^\sharp (t, x))>g^\prime(u^\sharp (t, y(t))).\]
This gives 
\[H(t, y(t))>g'(u^\sharp (t, y(t))). \]
A similar argument leads to
\[H(t, y(t)) < f^\prime(u^\flat (t, y(t))).\]

\noindent It remains to justify that $u^\flat(t,x)$ and $u^\sharp(t,x)$ stay arbitrarily close
to $u^-$ and $u^+$, respectively, for $t$ small and $x$ near~$0$.
This follows from the construction of $u^\flat$ and $u^\sharp$:
the function $u^\flat$ solves~\eqref{eq:sys_ub} with piecewise monotone initial data
$\bar u$ satisfying
$\displaystyle \lim_{x\to 0-}\bar u(x)=u^-$;
the function $u^\sharp$ is a rarefaction wave (hence continuous)
generated by initial data with $\displaystyle \lim_{x\to 0+}\bar u(x)=u^+$.
Therefore both traces $u^\flat(t,y(t)-)$ and $u^\sharp(t,y(t)+)$ remain close
to $u^-$ and $u^+$ for small~$t$, completing the proof.
\end{proof}

As a direct consequence, we can immediately prove the following one-side Lax admissibility inequality \emph{at $t=0$}.


\begin{corollary}\label{cor:2A}
    Let $y(t)$ be the solution to the Cauchy problem \eqref{eq:Yode}. Then
    \begin{align}
        \dot y(0+) = \lim_{\substack{t \to 0+ \\ y \ge g^\prime(u^+)t}} H(t, y(t))= \frac{f(u^-)-g(u^+)}{u^--u^+}=:\lambda> g^\prime(u^+).
    \end{align}
\end{corollary}
\begin{proof}
First, notice that, as an immediate consequence of \eqref{eq:lax-interface},
\[t \rightarrow u^\sharp (t, y(t)), \quad t \rightarrow u^\flat (t, y(t))\]
are monotone decreasing functions of time. Indeed, formally, assuming to rely on a regularization, we have
\[\frac{d}{dt} u^\flat(t, y(t))=u_t^\flat+\dot y(t) u_y^\flat =u_y^\flat (\dot y(t) - f^\prime(u^\flat))\le 0\] because $u_y^\flat \ge 0$. Similarly for $u^\sharp$. In particular, they are functions of bounded variation and the following limit is well defined:
\[
\lambda := \lim_{\substack{t \to 0+ \\ y \ge g^\prime(u^+)t}} H(t, y(t)).
\]
By construction of the solutions $u^\flat, u^\sharp$ to the Cauchy problems \eqref{eq:sys_ub}-\eqref{eq:sys_ud}, we deduce 
\[\dot y(0+)=\lambda=\frac{f(u^-)-g(u^+)}{u^--u^+}> g^\prime(u^+).\]
\end{proof}

We aim to show that $H(t,x(t))$ has bounded variation for small times. 
A key point in the proof is to ensure that the denominator in 
\eqref{eq:H} remains uniformly bounded away from zero. We prove this in the following. 

\subsubsection*{Proof of Proposition \ref{prop:main_2AI}}
For any $\varepsilon>0$, we consider the cone \begin{equation}\label{eq:cone}
    \Gamma_\varepsilon := \{\, (t,x) : t \ge 0,\ |x - \lambda t| \le \varepsilon\, t \,\}.
\end{equation}

    Consider the interface $y(t)$ that solves the Cauchy problem \eqref{eq:Yode}.
    Clearly, $y(0)$ belongs to the cone $\Gamma_\varepsilon$. For any $\varepsilon>0$,
    we can find $\delta>0$, a time $0<t_\delta<\min\{ t_1^\flat, t_1^\sharp\}$, where $t_1^\flat$ (resp. $t_1^\sharp$) is the first interaction time of $u^\flat$ (resp. $u^\sharp$) with the interface, and a constant $C>0$ such that \[|u^\flat(t, x) - u^-|<C\varepsilon, \quad  |u^\sharp(t, x) - u^+|<C\varepsilon, \quad |x| \le \delta, \quad  t\le t_\delta.\]
    This implies that
    \[|H(t, x)-\lambda| \le \varepsilon. \]
    In particular, the gluing solution evaluated at the interface $u(t, x(t))$ in \eqref{eq:sol_gluing} remains inside the cone $\Gamma_\varepsilon$ until a time $t_0 \le t_\delta$. Because $\lambda< f^\prime(u^-)$, we can choose $\varepsilon>0$ sufficiently small that 
    \begin{align}
        f^\prime(u^-)>\lambda + \varepsilon (2+C\|f''\|_{L^\infty}).
    \end{align}
    If the above inequality holds, then
    \[f^\prime(u^\flat) \ge f^\prime(u^-)-\|f''\|_{L^\infty}|u^\flat-u^-| \ge f^\prime(u^-)-\varepsilon C\|f''\|_{L^\infty} > \lambda + 2 \varepsilon.\]
    Assuming, additionally, that
    \[g^\prime(u^+)<\lambda-\varepsilon (2+C\|g''\|_{L^\infty}),\]
    yields
    \[g^\prime(u^+) < \lambda - 2\varepsilon.\]
    In particular, this implies that
    \[\dot x - f^\prime(u^\flat) \le \lambda + \varepsilon - \lambda - 2\varepsilon = - \varepsilon < 0,\]
    namely the function $t \rightarrow u^\flat (t)$ is monotonically decreasing because $u_x^\flat$ is positive. Similarly, we have that $\dot x - g^\prime(u^\sharp) > \varepsilon $ and $u_x^\sharp$ is decreasing, yielding that $t \rightarrow u^\sharp (t)$ is monotonically decreasing. Therefore, the total variation of $u^\flat (t)$ and $u^\sharp (t)$ is bounded by the total variation of the initial data $\bar u(x)$ in a neighborhood of the origin and \eqref{eq:TVbound_2AI} holds. 
    Moreover,
    \begin{align}\label{eq:lower_bound} u^\sharp - u^\flat \ge u^+-u^--|u^\flat-u^-|-|u^\sharp - u^+| >  u^+-u^--2C\varepsilon > \frac{u^+-u^-}{2}
    \end{align}
    for $\varepsilon>0$ sufficiently small. 
    We are left with proving the last item. Because
    \[f(u^\flat)-g(u^\sharp) \le (\|f^\prime\|_{L^\infty}+\|g^\prime\|_{L^\infty})(TV(u^\flat)+TV(u^\sharp)),\]
    where the latter is bounded by \eqref{eq:TVbound_2AI} and appealing to the lower bound \eqref{eq:lower_bound}, we deduce that $H(t, x)$ has bounded variation - with an estimate depending on the BV norm of the initial data $\bar u(x)$ - for $(t, x) \in \Gamma_\varepsilon$. \\

    \medskip
    In conclusion, we proved that $H(t, y(t))$ has \emph{directionally bounded variation} according to Definition \ref{def:localBV}.
    The proof of Case 2A is now a direct consequence of \cite{bressan1988}.
\subsubsection{Case 2B: $u^- < u^-_{\mathrm{lim}}$, $u^- < u^*$, and $u^+ < u^*$.}
Recall from \eqref{eq:sol2b} that the solution of the Riemann problem consists of a jump between $u^-$ and $u^*$ with speed 
\begin{equation}\label{eq:lambda2B}
    \lambda = g^\prime(u^*) = \frac{g(u^*)-f(u^-)}{u^*-u^-}, \quad \text{with} \quad u^-=\lim_{(t,x)\rightarrow(0,0)} u^\flat(t,x),
\end{equation}
followed by rarefaction a with flux $g(u)$ between states $u^*$ and $u^+$, with $g^\prime(u^*)<g^\prime(u^+)$. Clearly, the transversality property \eqref{eq:lax-interface} does not hold in this case where $\lambda=g^\prime(u^*)$, the theory of \cite{bressan1988} does not apply, and we are brought, following \cite{ABS2025}, to construct our solutions by a fixed point argument.

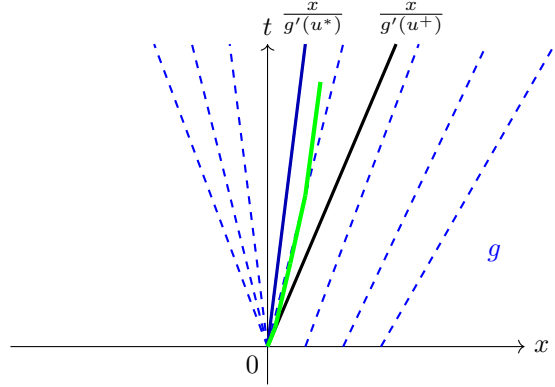
\begin{figure}[h!]
\centering
\begin{tikzpicture}[scale=1]
  \draw[->] (-3.4,0) -- (3.4,0) node[right] {$x$};
  \draw[->] (0,-0.5) -- (0,4) node[above] {$t$};
  
  \draw[thick, blue, dashed] (0,0) -- (-1.5,4);
  \draw[thick, blue, dashed] (0,0) -- (-1,4);
  \draw[thick, blue, dashed] (0,0) -- (-0.5,4);
  \draw[thick, blue, dashed] (0,0) -- (0.5,4);
  \draw[thick, blue, dashed] (0,0) -- (1,4);
  \draw[very thick, black] (0,0) -- (1.7,4);
  \node at (1.9,4) [above] {$\frac{x}{g^\prime(u^+)}$};
  \draw[very thick, blue!70!black] (0,0) -- (0.5,4);
  \node at (0.6,4) [above] {$\frac{x}{g^\prime(u^*)}$};
  \draw[thick, blue, dashed] (0.5,0) -- (2,4);
  \draw[thick, blue, dashed] (1,0) -- (2.9,4);
  \draw[thick, blue, dashed] (1.5,0) -- (3.8,4);
  \node at (3,1) [above] {\color{blue}$g$};
  \node at (-0.2,0) [below] {$0$};
  \draw[ultra thick, green] (0,0) -- (0.1,0.25);
  \draw[ultra thick, green] (0.1,0.25) -- (0.3,1.1);
  \draw[ultra thick, green] (0.3,1.1) -- (0.5,2);
  \draw[ultra thick, green] (0.5,2) -- (0.7,3.5);
\end{tikzpicture}
\caption{Interface location in Case 2B (green line) along with the characteristic curves of solution $u^\sharp(t,x)$.}\label{fig:solCase2b}
\end{figure}

To this end, we introduce the Picard operator as in \cite{ABS2025}:
\begin{equation}\label{eq:PicardOp2B}
    (\mathcal{P}y)(t) = \displaystyle\int_0^{T} \phi\left(u^\flat(s,y(s)-),u^\sharp(s,y(s))\right) ds \quad\mbox{ where }\quad \phi(u,v)=\frac{f(u)-g(v)}{u-v}.
\end{equation}

We aim at proving the following.
\begin{proposition}\label{prop:2B}
There exists \( t_0 > 0 \) sufficiently small such that the Picard operator \eqref{eq:PicardOp2B} is a contraction mapping on the complete metric space $(\mathcal{F}, d)$ where $\mathcal{F}$
is the family of Lipschitz functions
\begin{equation}
    \mathcal F = \left\{ y\in W^{1,\infty}([0,t_0]); \,|\dot y(t) -g^\prime(u^*) | \leq \varepsilon, \mbox{ for a.e. t}, \,\,\lim_{t\rightarrow 0}\frac{y(t)}{t}=g^\prime(u^*) \; \text{with} \; u^* \; \text{in} \; \eqref{eq:lambda2B} \right\},
\end{equation}
with distance 
\begin{equation}
    d(y,z) = \sup_{0<t\leq t_0} \left| \frac{y(t)-z(t)}{t}\right|.
\end{equation}
Moreover, the unique fixed point $\mathcal{P}(y(t))=y(t) \in \mathcal F$ is the unique solution of the initial value problem \eqref{eq:Yode} in \( [0,t_0] \), where $u^\flat(t, x)$, $ u^\sharp (t, x)$ solve the Cauchy problems \eqref{eq:sys_ub}-\eqref{eq:sys_ud} with $\bar u(0+)=u^*$.
\end{proposition}
The proof relies on the intermediate result below, whose proof is in \cite[Lemma 5.1]{ABS2025} and therefore we omit it.
\begin{lemma}\label{lemma2B}
Given a constant $a>0$, let $v:[0,\tau]\times\R \rightarrow \R$ be a function with properties:
\begin{itemize}
    \item[(i)] The function $x\rightarrow v(t,x)$ increases at all times.
    \item[(ii)] For every Lipschitz curve $t\rightarrow\gamma(t)$ with derivative $\dot\gamma(t)\ge a - \varepsilon$, the composite function $t\rightarrow v(t,\gamma(t))$ is decreasing. 
\end{itemize}
Consider any couple of Lipschitz functions $y,z:[0,\tau]\rightarrow \R$ such that 
\[
y(0)=z(0)=0, \quad \dot y(t) \ge a,  \ \dot z(t)\ge a \quad \mbox{for a.e. } t\in[0,\tau],
\]
and define 
\[
v_{max}=\sup_{t,x} v(t,x), \quad v_{min}=\inf_{t,x} v(t,x), \quad \delta(t) = \sup_{0<s<t}|y(s)-z(s)|.
\]
Then, for every $t\in[0,\tau]$ one has
\begin{equation}
    \int_0^t |v(s,y(s))-v(s,z(s)))| ds \leq 2 \delta(t)\,\frac{v_{max}-v_{min}}{\epsilon}.
\end{equation}
\end{lemma}
\begin{proof}[Proof of Proposition \ref{prop:2B}]
We first show that there exists $t_0>0$ such that the interface $y(t)$, solving \eqref{eq:Yode}, belongs to $\mathcal F$. 
For the centered rarefaction, where $x/t \leq g^\prime(u^+)$, we have $u^\sharp(t,x) = w$ with $x/t = g^\prime(w)$. In particular, $g^\prime(u^*) = \lambda$, and there exist $w^-_\varepsilon > u^* > w^+_\varepsilon$ such that
\[
g^\prime(w^-_\varepsilon) = \lambda - \varepsilon, \qquad g^\prime(w^+_\varepsilon) = \lambda + \varepsilon.
\]
\begin{itemize}
    \item If $w \geq w^-_\varepsilon> u^*$, by the concavity of $g$ and the inequality $f>g$,
    \[
    H = \frac{f(u^\flat)-g(w)}{u^\flat -w} > g^\prime(w),
    \]
    for $u^\flat(t, x)$ sufficiently close to $u^-$.
    Then, $\dot y > \lambda-\varepsilon$ choosing $\varepsilon$ sufficiently small.

    \item If $w \leq w^+_\varepsilon < u^*$, then $g^\prime(w) \geq g^\prime(w^+_\varepsilon) = \lambda+\varepsilon$ and, for $\varepsilon$ small enough,
    \[
    H = \frac{f(u^\flat)-g(w)}{u^\flat -w} < \lambda + \varepsilon.
    \]
\end{itemize}
This shows that $y(t)$ satisfies $|\dot y(t)-g^\prime(u^*)| \le \varepsilon$ for $t \in [0, t_0]$. 

\noindent Now, let $\mathfrak C:=u^*-u^-$. By the standard theory of conservation laws \cite{bressan2000} yielding the solutions $u^\flat (t, x), u^\sharp(t, x)$ to \eqref{eq:sys_ub}, \eqref{eq:sys_ud}, for any $0< \varepsilon < \frac{\mathfrak C}{4}$, there exist $t_{\varepsilon}>0$ and $\delta_{\varepsilon}>0$ such that

\begin{equation}
    \label{eq:epsilon}
    |u^\flat(t, x)-u^-|<\varepsilon, \quad |u^\sharp(t, x)-u^*|<\varepsilon \quad \text{for} \quad  t \le t_{\varepsilon}, \quad |x| \le \delta_{\varepsilon}.
\end{equation}
In particular, 
\[u^\flat - u^\sharp = u^\flat - u^- + u^- - u^* + u^*-u^\sharp > \frac{\mathfrak C}{2}, \quad t \in [0, t_{\varepsilon}],\]
namely the denominator of the integrand of the Picard operator $\mathcal P$ in \eqref{eq:PicardOp2B} is uniformly bounded from below for $t \in [0, t_{\varepsilon}]$. Consequently, choosing $t_0 \le t_{\varepsilon}$, we deduce that $\mathcal{P}:\mathcal F \rightarrow \mathcal F$. It remains to show that it is a contraction. To this end, let $y(t), z(t)$ be two Lipschitz functions with $d(y,z)=\delta$, so that
\[
|y(t)-z(t)|\leq \delta t, \quad  t\in[0,t_0].
\]
We write
\begin{equation}
(\mathcal{P}y)(t)-(\mathcal{P}z)(t) = \displaystyle\int_0^{T} A(s)ds + \int_0^{T} B(s)ds,     
\end{equation}
where
\begin{equation}\label{eq:A}
    A(s) = \phi\left(u^\flat(s,y(s)-),u^\sharp(s,y(s))\right)-\phi\left(u^\flat(s,z(s)-),u^\sharp(s,y(s))\right),
\end{equation}
\begin{equation}\label{eq:B}
    B(s) = \phi\left(u^\flat(s,z(s)-),u^\sharp(s,y(s))\right)-\phi\left(u^\flat(s,z(s)-),u^\sharp(s,z(s))\right).
\end{equation}
We start with $B(s)$. Applying the mean value theorem,
    \[
    B(s) = \phi_v\left(u^\flat(s,z(s)-),\widehat u^\sharp\right)\cdot\left(u^\sharp(s,y(s))-u^\sharp(s,z(s))\right),
    \]
where $\widehat u^\sharp$ lies between $u^\sharp(s,y(s))$ and $u^\sharp(s,z(s)))$.
Again, by the mean value theorem,
    \[
    \phi_v(u^\flat(s,z(s)-),\widehat u^\sharp) = \phi_v(u^-,u^*) + \phi_{vu}   (\widehat u^-,u^*)\left(u^\flat-u^-\right)+ \phi_{vv}(u^\flat, \widehat u^*)\left(\widehat u^\sharp-u^*\right),
    \]
    where $\widehat u^-, \widehat u^*$ lie between $u^-, u^\flat$ and $u^*, u^\sharp$ respectively and, by construction (see \eqref{eq:lambda2B}),
    \[
    \phi_v(u^-,u^*) = \frac{1}{u^--u^*}\left(\frac{f(u^-)-g(u^*)}{u^--u^*}-g^\prime(u^*)\right) = 0.
    \]
    Moreover, recalling from Assumptions \ref{assump} that $g''(u)>C$ for all $u \in [0, u_{\max}]$,
    \begin{align*}
    |u^\sharp(s,y(s))-u^\sharp(s,z(s))| = \left| (g^\prime)^{-1}\left(\frac{y(s)}{s}\right)-(g^\prime)^{-1}\left(\frac{z(s)}{s}\right)\right| \le  (\inf_u |g''(u)|)^{-1} \left|\frac{y(s)-z(s)}{s}\right| \le C d(y, z). 
    \end{align*}
    Putting everything together,
    \begin{align*}
        |B(s)| \le C\sup_{(u^\flat, u^\sharp) \, : \, |u-u^-|<\varepsilon, \, |v-u^*| < \varepsilon} |D^2 \phi(u, v)|(|u^\flat-u^-|+|u^\sharp-u^+|)  \,\, d(y, z) < 2 C \varepsilon \sup |D^2 \phi| d(y, z),
    \end{align*}
    where we choose $\varepsilon$ in \eqref{eq:epsilon} small enough that $C \varepsilon \sup |D^2 \phi|<\frac 14.$
    To estimate the term $A(s)$ in \eqref{eq:A}, we consider the function
\[
(w,s)\rightarrow\Phi(w,s)=\phi(w,u^\sharp(s,y(s)))=\frac{f(w)-g(u^\sharp(s,y(s)))}{w-u^\sharp(s,y(s))}, \quad 0<s<t.
\]
Notice that $\Phi$ is Lipschitz continuous with respect to $w$ and we denote the Lipschitz constant by $L$. Moreover, the function $u^\flat(t,x)$ defined in \eqref{eq:sys_ub} satisfies the assumptions of Lemma \ref{lemma2B}, for some $\varepsilon>0$.
Indeed, note that $u^\flat(\cdot, x)$ is increasing and, for any $y(t)$ (resp. $z(t)$) in $\mathcal{F}$, there holds: 
\[\dot y(t)-f'(u^\flat) < \dot y(t) - f'(u^*),\]
where we used that $u^\flat<u^*$ for $u^\flat$ sufficiently close to $u^-<u^*$. To show that $u^\flat(t, y(t))$ is a decreasing function of time, we need
\[\dot y - f'(u^*)<0.\]
However, $y(t) \in \mathcal F$ implies
\[\dot y - f'(u^*) < g'(u^*)+\varepsilon-f'(u^*) = \frac{g(u^*)-f(u^-)}{u^*-u^-}-f'(u^*)+\varepsilon<\frac{f(u^*)-f(u^-)}{u^*-u^-}-f'(u^*)+\varepsilon<0\]
for $\varepsilon>0$ sufficiently small. Moreover, it obviously holds that $\dot y\ge g'(u^*)-\varepsilon$.
Using Lemma \ref{lemma2B} we can thus estimate 
\begin{align}
    \int_0^{T}|A(s)|ds &=\int_0^{T} |\Phi(u^\flat(s,y(s)-),s) - \Phi(u^\flat(s,z(s)-))| ds
    \\
    &\leq L\cdot\int_0^{T} |u^\flat(s,y(s)-),s)-u^\flat(s,z(s)-)| ds
    \\
    & \leq \frac{L}{\varepsilon}(u^\flat_{max}-u^\flat_{min})\sup_{0<s<t}|y(t)-z(t)|
    \\
    &\leq \frac{L}{\varepsilon}(u^\flat_{max}-u^\flat_{min})\cdot td(y,z).
\end{align}
We can now choose $\varepsilon$ in \eqref{eq:epsilon} such that
\[
u^\flat_{max}-u^\flat_{min} \leq \frac{\epsilon}{4L},
\]
yielding, for $t \in [0, t_0]$
\begin{equation*}
    \int_0^{T}|A(s)|ds \leq \frac{t}{4}d(y,z)<\frac 1 4 d(y,z).
\end{equation*}
\end{proof}


\subsubsection{Case 2C: $u^- > u^-_{\mathrm{lim}}$.}
Following the construction in \eqref{eq:case2C}, we modify the initial datum of $u^\flat$ in \eqref{eq:sys_ub} as
\begin{align}
u^\flat(0,x)&=\begin{cases}
\bar u(x), \quad x<0,\\
u_{\max}, \quad x>0.
\end{cases}
\end{align}
Since $u^-=\lim_{x\to0-}\bar u(x)$, for $t>0$ the solution $u^\flat$ develops a shock at $x=\lambda t$, with
\[
\lambda=\frac{f(u_{\max})-f(u^-)}{u_{\max}-u^-}<0,
\]
and $u^\flat(t,x)=u_{\max}$ for $x>\lambda t$.

In the region $x<\lambda t$, one has $u^\flat(t,x)\leq u^-$. Since $\lambda<g'(u_{\max})$, we have $u^\sharp=u_{\max}$ there, and hence $u^\flat<u^\sharp$. By the concavity of $f$ and the identity $f(u_{\max})=g(u_{\max})$,
\begin{align}
\lambda
&=\frac{f(u_{\max})-f(u^-)}{u_{\max}-u^-}
<\frac{f(u_{\max})-f(u^\flat)}{u_{\max}-u^\flat}
<\frac{g(u_{\max})-f(u^\flat)}{u^\sharp-u^\flat}
<\frac{g(u^\sharp)-f(u^\flat)}{u^\sharp-u^\flat}
=H(t,x).
\end{align}
Thus $\dot y(t)>\lambda$. In particular, the solution $y(t)$ of \eqref{eq:Yode}, starting from the origin, immediately enters the region $x>\lambda t$, where $u^\flat\equiv u_{\max}$ (see Figure \ref{fig:solCase2c}).

The existence and uniqueness of a solution to \eqref{eq:Yode} is then straightforward. Indeed, $H(t,y)$ is uniformly Lipschitz in $y$ for sufficiently small times, so the Cauchy--Lipschitz theorem applies. To see this, for every $\varepsilon>0$, there exist $\delta_\varepsilon>0$ and $t_\varepsilon>0$ such that
\[
|u_{\max}-u^\sharp(t,x)|<\varepsilon,
\qquad |x|\leq\delta_\varepsilon,\quad
t\in[0,t_\varepsilon].
\]
Since $u^\flat=u_{\max}$ and $f(u_{\max})=g(u_{\max})$, we have
\[
H(t,y)
=\frac{g(u_{\max})-g(u^\sharp(t,y))}
{u_{\max}-u^\sharp(t,y)}
=g'(u_{\max})+\mathcal{O}(\varepsilon),
\]
uniformly for $|y|\leq\delta_\varepsilon$ and $t\in[0,t_\varepsilon]$.
Moreover, since $u^\sharp(t,\cdot)$ is a rarefaction wave,
\[
|u^\sharp(t,y)-u^\sharp(t,z)|\lesssim |y-z|.
\]
The mean value theorem therefore gives
\[
\begin{aligned}
|H(t,y)-H(t,z)|
&\lesssim |u^\sharp(t,y)-u^\sharp(t,z)|
\lesssim |y-z|,
\end{aligned}
\]
for $|y|,|z|\leq\delta_\varepsilon$ and $t\in[0,t_\varepsilon]$. Hence $H$ is uniformly Lipschitz in $y$ in this region, and the Cauchy--Lipschitz theorem yields the existence and uniqueness of the solution to \eqref{eq:Yode} for sufficiently small times.
\begin{figure}[h!]
\centering
\begin{tikzpicture}[scale=1]
  \draw[->] (-3.4,0) -- (3.4,0) node[right] {$x$};
  \draw[->] (0,-0.5) -- (0,4) node[above] {$t$};

  \draw[very thick, black] (0,0) -- (-2.5,4);
  \node at (-2.5,4) [above] {$\frac{x}{\lambda}$};
  \draw[very thick, blue] (0,0) -- (-1.5,4);
  \node at (-1.3,4) [above] {$\frac{x}{g^\prime(u_{\max})}$};
  \draw[thick, blue, dashed] (0,0) -- (-1,4);
  \draw[thick, blue, dashed] (0,0) -- (-0.5,4);
  \draw[thick, blue, dashed] (0,0) -- (0.5,4);
  \draw[thick, blue, dashed] (0,0) -- (1,4);
  \draw[very thick, blue] (0,0) -- (1.7,4);
  \node at (1.9,4) [above] {$\frac{x}{g^\prime(u^+)}$};
  \draw[thick, blue, dashed] (0,0) -- (0.5,4);
  \draw[thick, blue, dashed] (0.5,0) -- (2,4);
  \draw[thick, blue, dashed] (1,0) -- (2.9,4);
  \draw[thick, blue, dashed] (1.5,0) -- (3.8,4);
  \node at (3,1) [above] {\color{blue}$g$};
  \node at (-0.2,0) [below] {$0$};
  \draw[ultra thick, green] (0,0) -- (-0.05,0.25);
  \draw[ultra thick, green] (-0.05,0.25) -- (-0.3,1.1);
  \draw[ultra thick, green] (-0.3,1.1) -- (-0.5,2);
  \draw[ultra thick, green] (-0.5,2) -- (-0.7,3.5);
\end{tikzpicture}
\caption{Interface location in Case 2C (green line) along with the characteristic curves (blue lines) of solution $u^\sharp(t,x)$ and the shock curve $x=\lambda\,t$ (black line).}\label{fig:solCase2c}
\end{figure}
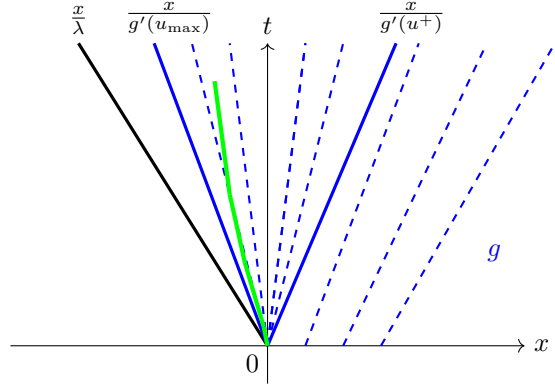

\medskip

The remaining cases, Case 3 in Section \ref{sec:case3} and Case 4 in Section \ref{sec:case4}, can be addressed using the same techniques.

\subsection{Case 3. $\theta^-=0, \theta^+=1$ ($g \rightarrow f$)}\label{sec:case3solution}
As in the previous case, we consider two solutions:
\begin{itemize}
\item The solution $u=u^\flat(t,x)$ of the Cauchy problem
    \begin{equation}\label{eq:sys_ubC3}
    \begin{array}{cc}
    u_t+g(u)_x=0, 
    &
    u(0,x)=\begin{cases}
    \bar u(x) \mbox{ if } x<0,\\
    0 \mbox{ if } x>0.
    \end{cases}
    \end{array}
    \end{equation}
    Since $\bar u$ is decreasing for $x<0$, the solution $u^\flat$ will contain only rarefaction fronts, together with a rarefaction fan centered at the origin. We recall from Section \ref{sec:case3} that there are three different cases 3A, 3B, 3C. In particular, the initial profile $\bar u$ satisfies:  
    \[
    \bar u(0-)=
    \begin{cases}
        u^-, & \text{in Case~3A},\\
        v^*, & \text{in Case~3B},\\
        0.   & \text{in Case~3C}.
    \end{cases}
\]
\item The solution $u=u^\sharp(t,x)$ of the Cauchy problem
    \begin{equation}\label{eq:sys_udC3}
    \begin{array}{cc}
    u_t+f(u)_x=0, 
    \end{array}
    \end{equation}
    where, in Case 3A and 3B, the initial profile 
    \begin{equation}
    u(0,x)=\begin{cases}
    u^+ \mbox{ if } x<0,\\
    \bar u(x) \mbox{ if } x>0,
    \end{cases}
    \end{equation}
    with \[\ u^+=\lim_{x\rightarrow 0+}\bar u(x),\]
    while in Case 3C
    \[u^\sharp(0,x)=\begin{cases}
    0 \mbox{ if } x<0, \\
    \bar u(x) \mbox{ if } x>0.
    \end{cases}\]
    Since $\bar u$ increases in $x>0$, $u^\sharp$ contains shocks and compression waves.
\end{itemize}

The full solution $u=u(t,x)$ is obtained by gluing these two solutions
across an interface $y=y(t)$, which is required to satisfy the ODE
\begin{equation}\label{eq:case3ode}
    \dot y(t)=H(t,y(t)),
    \qquad
    H(t,y(t))
    =\frac{g\bigl(u^\flat(t,y(t)-)\bigr)
    -f\bigl(u^\sharp(t,y(t)+)\bigr)}
    {u^\flat(t,y(t)-)-u^\sharp(t,y(t)+)}.
\end{equation}

The existence and uniqueness of the solution to \eqref{eq:case3ode} follow
from the same arguments as in Case~2. We therefore omit the details.

\subsection{Case 4. $\theta^-=1, \theta^+=1$ ($f \rightarrow f$)}
We distinguish three sub-cases.
\subsubsection{Case 4A: $u^-\leq u^+$.} In this case, $u=u(t,x)$ is the solution to the scalar conservation law \eqref{eq:CLf} with the the initial datum $\bar u$, increasing on both $x<0$ and $x>0$, and with $\theta(t,x)\equiv 1$ for all $t,x$.
\subsubsection{Case 4B: $u^->u^+$ with $u^-<u^-_{\mathrm{lim}}$ and $u^+>u^+_\mathrm{lim}$.}
Motivated by the solution of the Riemann problem
in Section \ref{sec:case4}, in this case two interfaces are generated by the initial discontinuity at t = 0. Specifically, we consider two solutions:
\begin{itemize}
\item The solution $u=u^\flat(t,x)$ of the Cauchy problem
    \begin{equation}\label{eq:sys_ub4}
    \begin{array}{cc}
    u_t+f(u)_x=0, 
    &
    u(0,x)=\left\{\begin{array}{l}
    \bar u(x) \mbox{ if } x<0,\\
    u^- \mbox{ if } x>0.
    \end{array}\right.
    \end{array}
    \end{equation}
\item The solution $u=u^\sharp(t,x)$ of the Cauchy problem
    \begin{equation}\label{eq:sys_ud4}
    \begin{array}{cc}
    u_t+f(u)_x=0, 
    &
    u(0,x)=\left\{\begin{array}{l}
    u^+ \mbox{ if } x<0,\\
    \bar u(x) \mbox{ if } x>0.
    \end{array}\right.
    \end{array}
    \end{equation}
\end{itemize} 
 The initial datum \(\bar u\) is increasing, and both solutions are monotonically increasing, consisting only of shocks and compression waves.
 We also consider the solution to \eqref{eq:CLg} consisting of a single centered rarefaction wave. This can be implicitly defined by setting
\begin{equation}
    u^\natural(t,x)=w \quad\mbox{if}\quad x/t = g'(w).
\end{equation}
Following \cite{ABS2025}, we then construct a solution to
the Cauchy problem by setting
\begin{equation}
    u(t,x) = \begin{cases}
        u^\flat(t,x) & \mbox{if } x < y(t),
        \\
        u^\natural(t,x) & \mbox{if } y(t) < x < z(t),
        \\
        u^\sharp(t,x) & \mbox{if } x > z(t),
    \end{cases}
\end{equation}
where the two interfaces $y(\cdot)$ and $z(\cdot)$ are uniquely determined by solving the ODEs
\begin{align}
    \dot y(t) =\frac{g(u^\natural(t,y(t)))-f(u^\flat(t,y(t)))}{u^\natural(t,y(t))-u^\flat(t,y(t))}, & & y(0) = 0,
    \\
    \dot z(t) =\frac{g(u^\natural(t,z(t)))-f(u^\sharp(t,z(t)))}{u^\natural(t,z(t))-u^\sharp(t,z(t))}, & & z(0) = 0.
\end{align}
The existence and uniqueness of the solutions $y$ and $z$ is proved by the same arguments as in Case 2B and Case 3B, respectively.

\subsubsection{Case 4C: $u^->u^+$ with $u^->u^-_{\mathrm{lim}}$ and/or $u^+<u^+_\mathrm{lim}$.}
This case can be treated as in Case 4B. However, if $u^->u^-_{\mathrm{lim}}$, the existence and uniqueness of the interface $y=y(t)$ follow from the same arguments used in Case 2C. Similarly, if $u^+<u^+_{\mathrm{lim}}$, the existence and uniqueness of $z=z(t)$ follow from the same arguments as in Case 3C.

\bigskip

Finally, the existence of solutions follows by a standard bootstrap argument, as in \cite{ABS2025}.


\section{Two-Phase Traffic Model}\label{sec:2phase}
In this section, we consider the situation where $f(u)=g(u)$ in the free-flow regime. 
\begin{figure}[h!]
\centering
\begin{tikzpicture}[scale=1.0]
  \draw[->] (-0.5,0) -- (5,0) node[right] {};
  \draw[->] (0,-0.5) -- (0,3.5) node[above] {};
  \draw[domain=0:4, smooth, variable=\x, red, thick] 
       plot ({\x}, {0.8*\x*(4-\x)});
  \draw[domain=0:1, smooth, variable=\x, blue, thick] 
       plot ({\x}, {0.8*\x*(4-\x)});    
  \draw[domain=1:4, smooth, variable=\x, blue, thick] 
       plot ({\x}, {2.4+1.6*(\x-1)-1.1*(\x-1)*(\x-1)+0.1*(\x-1)*(\x-1)*(\x-1)});
  \node at (2.8,2.8) [above] {\color{red}$f$};
  \node at (2.8,1.5) [above] {\color{blue}$g$};
  \node at (-0.2,0) [below] {$0$};
  \node at (4,0) [below] {$u_{\max}$};
  \draw[thick, black, dashed] (1.3,0) -- (1.3,2.8);
  \node at (1.3,0) [below] {$u_{\mathrm{free}}$};
\end{tikzpicture}
\begin{tikzpicture}[scale=1.0]
  \draw[->] (-0.5,0) -- (5,0) node[right] {};
  \draw[->] (0,-0.5) -- (0,3.5) node[above] {};
  \draw[domain=0:4, smooth, variable=\x, red, thick] 
       plot ({\x}, {0.8*\x*(4-\x)});
  \draw[domain=0:0.8, smooth, variable=\x, blue, thick] 
       plot ({\x}, {0.8*\x*(4-\x)});    
  \draw[domain=0.8:4, smooth, variable=\x, blue, thick] 
       plot ({\x}, {0.5*(\x+0.455)*(4-\x)});
  \node at (2.8,2.8) [above] {\color{red}$f$};
  \node at (2.5,1.5) [above] {\color{blue}$g$};
  \node at (-0.2,0) [below] {$0$};
  \node at (4,0) [below] {$u_{max}$};
  \draw[thick, black, dashed] (0.8,0) -- (0.8,2);
  \node at (0.8,0) [below] {$u_{\mathrm{free}}$};
   \draw[thick, black] (0.8,{0.8*0.8*(4-0.8)}) -- (2.1,{0.8*2.1*(4-2.1)});
   \node at (2.1,0) [below] {$\widehat{u}_{\mathrm{free}}$};
   \draw[thick, black, dashed] (2.1,0) -- (2.1,{0.8*2.1*(4-2.1)});
\end{tikzpicture}
\caption{Fundamental diagram for the two-phase traffic model considered in Section \ref{sec:2phase}.}\label{fig:freeFlow}
\end{figure}
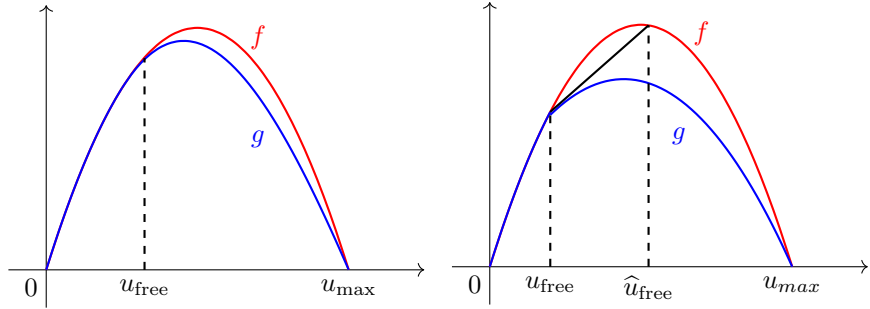
\begin{assumptions}\label{assump-Two-Phase Model: }
Let $u_{\max}>0$ and let
$u_{\mathrm{free}}\in(0,u_{\max})$ denote the transition density. Assume that
\[
f\in C^2([0,u_{\max}])
\]
is a strictly concave flux function satisfying
\[
f(0)=f(u_{\max})=0.
\]
Assume that
\[
g\in C([0,u_{\max}])\cap
C^2\!\left([0,u_{\max}]\setminus\{u_{\mathrm{free}}\}\right)
\]
is a concave flux function satisfying
\[
g(0)=g(u_{\max})=0.
\]
Finally, assume that the two fluxes coincide throughout the free-flow regime, namely,
\[
f(u)=g(u),
\qquad
\forall\,u\in[0,u_{\mathrm{free}}].
\]
\end{assumptions}


\begin{remark}[Regularity of $g$]
Since $f=g$ on $[0,u_{\mathrm{free}}]$, we necessarily have
\[
f'(u_{\mathrm{free}})
=
g'(u_{\mathrm{free}}-).
\]
If, in addition,
\[
g'(u_{\mathrm{free}}-)
=
g'(u_{\mathrm{free}}+),
\]
that is, if $g\in C^1([0,u_{\max}])$ (see Figure~\ref{fig:freeFlow}, left), then the construction of the Riemann solvers and the proof of existence of weak solutions remain essentially unchanged.
\end{remark}


Unless otherwise stated, we allow $g$ to have a corner at $u_{\mathrm{free}}$ and assume
\[
f'(u_{\mathrm{free}})
=
g'(u_{\mathrm{free}}-)
>
g'(u_{\mathrm{free}}+).
\]

\subsection{Two-Phase Model:  Case 1 ($\theta^-=\theta^+=0$, $g\rightarrow g$).}\label{sec:freeflow_case1}
We consider the conservation law
\[
u_t+(g(u))_x=0,
\]
with Riemann initial data
\[
u(0,x)=
\begin{cases}
u^-, & x<0,\\
u^+, & x>0.
\end{cases}
\]

The case $u^-<u^+$ is identical to Case~1 in the standard configuration (see Section~\ref{sec:case1}) and therefore will not be discussed further.

Assume now that
\[
u^+<u_{\mathrm{free}}<u^-.
\]
In the standard configuration, the corresponding solution consists of a centered rarefaction wave. In the present setting, however, an additional difficulty arises from the jump in the characteristic speed at the transition density,
\[
g'(u_{\mathrm{free}}-)
>
g'(u_{\mathrm{free}}+).
\]

The corresponding Riemann solver is
\begin{align}\label{eq:rar_disc}
u(t,x)=
\begin{cases}
w, &
\dfrac{x}{t}=g'(w),
\qquad
g'(u^-)\le\dfrac{x}{t}\le g'(u_{\mathrm{free}}{+}),
\\[2mm]
u_{\mathrm{free}}, &
g'(u_{\mathrm{free}}+)
\le
\dfrac{x}{t}
\le
g'(u_{\mathrm{free}}-),
\\[2mm]
w, &
\dfrac{x}{t}=g'(w),
\qquad
g'(u_{\mathrm{free}}{-})\le\dfrac{x}{t}\le g'(u^+).
\end{cases}
\end{align}

Hence, the centered rarefaction is split into two rarefaction fans separated by a constant intermediate state equal to $u_{\mathrm{free}}$. In particular, the solution is constant in the region
\[
g'(u_{\mathrm{free}}+)
\le
\frac{x}{t}
\le
g'(u_{\mathrm{free}}-).
\]


\subsection{Two-Phase Model:  Case 2 ($\theta^-=1$, $\theta^+=0$, $f\rightarrow g$).}\label{sec:freeflow_case2}
\begin{itemize}
\item If $u^->u_{\mathrm{free}}$, then the solution coincides with that of the corresponding case in the standard configuration (see Section~\ref{sec:case2}).

\item Let $u^-\leq u_{\mathrm{free}}$. The solution can then be treated exactly as in Case 1 (Section~\ref{sec:freeflow_case1}), namely by considering the entropy solution $u=u(t,x)$ of the scalar conservation law
\[
u_t+(g(u))_x=0
\]
with the same initial data.

More precisely:
\begin{itemize}
    \item if $u^-\leq u^+$, the solution consists of a single shock propagating with speed
    \[
    \lambda=\frac{g(u^+)-f(u^-)\equiv g(u^-)}{u^+-u^-},
    \]
    and we get 
    \[
        (u,\theta)(t,x) = \left\{\begin{array}{cc}
              (u^-,1) &  x/t < \lambda, \\
              (u^+,0) &  x/t > \lambda.
        \end{array}\right. 
    \]
    
    \item if $u^+<u^-<u_{\mathrm{free}}$, 
    the solution is a monotonically decreasing centered rarefaction connecting $u^-$ to $u^+$. In this case $u_x(t,x)\leq 0$ for all $(t,x)$, and therefore the choice $\theta(t,x)\equiv 0$ is admissible.
\end{itemize}
\end{itemize}

\subsection{Two-Phase Model:  Case 3 ($\theta^-=0$, $\theta^+=1$, $g\rightarrow f$).}\label{sec:freeflow_case3}
We introduce the value $\widehat{u}_{\mathrm{free}}$ such that
\begin{equation}
    \widehat{u}_{\mathrm{free}} > u_{\mathrm{free}} \quad\text{with}\quad g^\prime(u_{\mathrm{free}}+)=\frac{f(\widehat{u}_{\mathrm{free}})-g(u_{\mathrm{free}})}{\widehat{u}_{\mathrm{free}}-u_{\mathrm{free}}},
\end{equation}
see Figure \ref{fig:freeFlow}, right.
For $u^+ < \widehat{u}_{\mathrm{free}}$, the value $v^*$ given in \eqref{eq:vstar} is undefined.
We then consider three different scenarios:
\begin{itemize}
    \item If $u^+\leq u_{\mathrm{free}}$ then the solution is as \emph{Two-Phase Model:  Case 1 ($g\rightarrow g$)} given in Section \ref{sec:freeflow_case1}.
    \item If $u^+\geq \widehat{u}_{\mathrm{free}}$ then the solution of the Riemann problem remains unchanged with respect to the same-named case in the \textit{standard} configuration, given in Section \ref{sec:case3}.
    \item If $u_{\mathrm{free}}<u^+<\widehat{u}_{\mathrm{free}}$ then the point $v^*$ in \eqref{eq:vstar} is not defined, and we have to distinguish the following cases:
    \begin{itemize}
        \item for $u^-\leq u_{\mathrm{free}}$ the solution consists of a Lax-admissible single jump with 
        $$
        \lambda = \frac{f(u^+)-g(u^-)\equiv f(u^-)}{u^+-u^-}.
        $$
        \[
        (u,\theta)(t,x) = \left\{\begin{array}{cc}
             (u^-,0)     & x/t < \lambda \\
             (u^+,1)  &  x/t > {\lambda}.
        \end{array}\right.
        \]
        \item for $u^->u_{\mathrm{free}}$ the solution is given by a rarefaction, followed by a constant state and later by an upward jump:
        \begin{equation}\label{eq:sol3}
        (u,\theta)(t,x) = \left\{\begin{array}{cc}
             (u^-,0)     & x/t < g^\prime(u^-), \\
             (w,0)       &  x/t=g'(w),\ g^\prime(u^-) \le x/t \le g'(u_{\mathrm{free}}+),\\
             (u_{\mathrm{free}}, 0)       &  \; g'(u_{\mathrm{free}}+) < x/t < \lambda\\
             (u^+, 1)  &  x/t > \lambda.
        \end{array}
        \right.   
        \end{equation}
        where $\lambda$ is the speed of the jump from $u_{\mathrm{free}}$ to $u^+$, i.e.
        $$
        \lambda = \frac{f(u^+)-g(u_{\mathrm{free}})\equiv f(u_{\mathrm{free}})}{u^+-u_{\mathrm{free}}}.
        $$
       
        \begin{remark} Solution \eqref{eq:sol3} is \textbf{not unique}. Unlike Case 2C and Case 3C, where the admissible location of the interface was constrained, in the present case several choices are admissible. We could move the interface (from $\lambda$ to $g'(u_{\mathrm{free}}+)$):
         \begin{equation}\label{eq:sol3_2}
        (u,\theta)(t,x) = \left\{\begin{array}{cc}
             (u^-,0)     & x/t < g^\prime(u^-), \\
             (w,0)       &  x/t=g'(w),\ g^\prime(u^-) \le x/t \le g'(u_{\mathrm{free}}+),\\
             (u_{\mathrm{free}}, 1)       &  \; g'(u_{\mathrm{free}}+) < x/t < \lambda=\frac{f(u^+)-g(u_{\mathrm{free}})}{u^+-u_{\mathrm{free}}},\\
             (u^+, 1)  &  x/t > \lambda.
        \end{array}
        \right.   
        \end{equation}
        Alternatively, we could consider:
        \begin{equation}\label{eq:sol3_3}
        (u,\theta)(t,x) = \left\{\begin{array}{cc}
             (u^-,0)     & x/t < g^\prime(u^-), \\
             (w,0)       &  x/t=g'(w),\ g^\prime(u^-) \le x/t \le g'(u_{\mathrm{free}}+),\\
             (u_{\mathrm{free}}, 0)       &  \; g'(u_{\mathrm{free}}+) < x/t < g'(u_{\mathrm{free}}-),\\
             (w, 0)       &  x/t=g'(w),\ g'(u_{\mathrm{free}}-) \le x/t \le g'(0),\\
             (u^+, 1)  &  x/t > \lambda=\frac{f(u^+)}{u^+}.
        \end{array}
        \right.   
        \end{equation}
        The non-uniqueness stems from the discontinuity of $g'$ at $u_\mathrm{free}$, since the plateau at $u=u_\mathrm{free}$ is included in the definition of the rarefaction wave \eqref{eq:rar_disc}. To select a physically meaningful interface, based on traffic modeling considerations, we assumed in \eqref{eq:sol3} that the transition from the $g$-flux regime (acceleration) to the $f$-flux regime (deceleration) propagates the Rankine--Hugoniot speed $\lambda$ connecting $u_\mathrm{free}$ and $u^+$.
        \end{remark}
        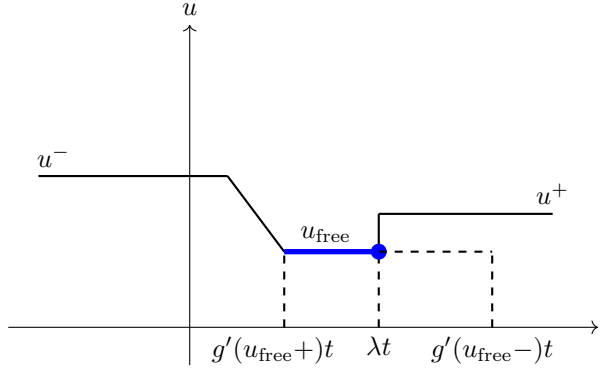
\begin{figure}[h!]
        \centering
        \begin{tikzpicture}[scale=1.0]
          \draw[->] (-3.4,0) -- (4.4,0) node[right] {};
          \draw[->] (-1,-0.5) -- (-1,4) node[above] {$u$};
          
          \draw[thick, black] (-3,2) -- (-0.5,2);
          \draw[thick, black] (-0.5,2) -- (0.25,1);
          \draw[thick, black] (0.25,1) -- (1.5,1);
          \draw[thick, black] (1.5,1) -- (1.5,1.5);
          \draw[thick, black] (1.5,1.5) -- (3.8,1.5);
          
          \node at (0.1,0) [below] {$g^\prime(u_\mathrm{free}+) t$};
          \draw[thick, black, dashed] (0.25,0) -- (0.25,1);
          
          \node at (1.5,0) [below] {$\lambda t$};
          \draw[thick, black, dashed] (1.5,0) -- (1.5,1);
          \node at (-2.8,2) [above] {$u^-$};
          \node at (3.8,1.5) [above] {$u^+$};
          \node at (0.8,1) [above] {$u_{\mathrm{free}}$};
          \fill[blue] (1.5,1) circle (3pt);
          \draw[line width=2pt, blue] (0.25,1) -- (1.5,1);
          \draw[thick, black, dashed] (1.5,1) -- (3,1);
          \draw[thick, black, dashed] (3,0) -- (3,1);
          \node at (3,0) [below] {$g^\prime(u_\mathrm{free}-) t$};
        \end{tikzpicture}
        \caption{Representation of solution \eqref{eq:sol3}.}\label{fig:RScase4}
        \end{figure}
    \end{itemize}
\end{itemize}

\subsection{Two-Phase Model:  Case 4 ($\theta^-=1$, $\theta^+=1$, $f\rightarrow f$). }
\begin{itemize}
\item If $u^-$ and $u^+$ are both less or equal $u_{\mathrm{free}}$ then the case can be treated as \emph{Two-Phase Model:  Case 1 ($g\rightarrow g$)}, Section \ref{sec:freeflow_case1}; 
\item if $u^-\leq u_{\mathrm{free}}$ and $u^+>u_{\mathrm{free}}$ then we fall back into \emph{Two-Phase Model:  Case 3 ($g\rightarrow f$)}, Section \ref{sec:freeflow_case3}; 
\item if $u^->u_{\mathrm{free}}$ and $u^+\leq u_{\mathrm{free}}$ then we fall back into \emph{Two-Phase Model:  Case 2 ($f\rightarrow g$)}, Section \ref{sec:freeflow_case2}.
\item if $u^->u_{\mathrm{free}}$ and $u^+>u_{\mathrm{free}}$ then fall back into the Case 4 ($f\rightarrow f$) in the standard configuration, Section \ref{sec:case4}.
\end{itemize}

\subsection{Existence of weak solutions to the Two-Phase Model}
We only emphasize the cases where the proof differs from the standard case under Assumption \ref{assump}.

\subsubsection{Two-Phase Model:  Case 2 with $u^-\leq u_{\mathrm{free}}$ and $u^-<u^+$.}
Referring to the Riemann Solver given in Section \ref{sec:freeflow_case2}, we assume
$$
\theta(t,x)=\left\{\begin{array}{cc}
     1 &  x\leq y(t),\\
     0 &  x> y(t),
\end{array}\right.
$$
where the interface verifies
\begin{equation}\label{eq:Yode2fr}
    \dot y(t)=H(t,y(t)), \quad y(0)=0, \quad H(t,x) = \displaystyle\frac{f(u^\flat(t,x-))-g(u^\sharp(t,x+))}{u^\flat(t,x-)-u^\sharp(t,x+)},
\end{equation} 
with $u^\flat(t,x)$ and $u^\sharp(t,x)$ solutions of \eqref{eq:sys_ub} and \eqref{eq:sys_ud} respectively.

In this case, the function $u^\flat(t,x)$ develops shocks and is non-decreasing. Since $\bar u(0-)=u^-$, there exists $\delta>0$ such that $u^\flat(t,x)\leq u^-$ for $t+|x|<\delta$. Moreover, since we assume $u^-\leq u_{\mathrm{free}}$, in this region we have $f(u^\flat(t,x))\equiv g(u^\flat(t,x))$ and   
\begin{equation}
    H(t,y(t)) = \displaystyle\frac{g(u^\flat(t,y(t)-))-g(u^\sharp(t,y(t)+))}{u^\flat(t,y(t)-)-u^\sharp(t,y(t)+)},
\end{equation}
which is locally Lipschitz continuous. Indeed, if $u_{\mathrm{free}}$ lies between $u^\flat(t,y(t)-)$ and $u^\sharp(t,y(t)+)$, then the denominator remains strictly bounded away from zero, ensuring the local Lipschitz continuity of the quotient. On the other hand, if $u_{\mathrm{free}}$ does not belong to the interval between the two states, then $g$ is locally $C^2$ in a neighborhood of the relevant values, and the regularity of $H$ follows directly.

\subsubsection{Two-Phase Model:  Case 3 with $u_{\mathrm{free}}<u^+\leq \hat u_{\mathrm{free}}$.}
As in Case 3 of Section \ref{sec:case3solution}, we consider the two Cauchy problems \eqref{eq:sys_ubC3} and \eqref{eq:sys_udC3}. In the present case, however, the solution to \eqref{eq:sys_udC3} is given by \eqref{eq:rar_disc}. Following the same argument as above, we obtain the solution $u=u(t,x)$ to the Cauchy problem associated with \eqref{eq:original} by gluing the two solutions across the interface $y=y(t)$, which satisfies
\begin{equation}\label{eq:Yode3fr}
    \dot y(t)=H(t,y(t)), \quad y(0)=0, \quad H(t,y(t)) = \frac{g(u^\flat(t,y(t)-))-f(u^\sharp(t,y(t)+))}{u^\flat(t,y(t)-) - u^\sharp(t,y(t)+)}.
\end{equation}
We distinguish the following two cases.
\begin{itemize}
    \item $u^- < u_{\mathrm{free}}$. 
    By construction, $\bar u(0-)=u^-$. Due to the continuity of $u^\flat(t,x)$, there exists $\bar\xi$ such that $u^\flat(t,x)=u_{\mathrm{free}}$ along the line $x=\bar\xi+g'(\bar u(\bar\xi))t$. Furthermore, for every $\epsilon>0$, there exists $\delta>0$ such that
    \[
    |u^\flat(t,x) - u^-| < \epsilon \quad \text{whenever } t \leq (x-\bar \xi)/g'(u_{\mathrm{free}}-), \ t+|x| < \delta.
    \]
    Within this region, we have $u^\flat(t,x) \leq u_{\mathrm{free}}$, which implies that $g(u^\flat(t,x)) = f(u^\flat(t,x))$. Consequently, $H(t,x)$ takes the form
    \begin{equation}\label{eq:H3fr}
    H(t,x) = \displaystyle\frac{f(u^\flat(t,x-))-f(u^\sharp(t,x+))}{u^\flat(t,x-)-u^\sharp(t,x+)},
    \end{equation}
    which is locally Lipschitz continuous owing to the regularity of $f$.
\begin{figure}[h!]
\centering
\begin{tikzpicture}[scale=0.8]
  \draw[->] (-3.4,0) -- (3.4,0) node[right] {$x$};
  \draw[->] (0,-0.5) -- (0,4) node[left] {$t$};
  
  \draw[thick, blue, dashed] (0,0) -- (2.4,4);
  \draw[thick, blue, dashed] (0,0) -- (3,4);
   \draw[thick, blue, dashed] (0,0) -- (3.6,4);
   \node at (3.9,4) [above] {$\frac{x}{g^\prime(0)}$};
  \draw[thick, blue, dashed] (0,0) -- (2,4);
  \node at (2.2,4) [above] {$\frac{x}{g^\prime(u^-)}$};
  \draw[thick, blue, dashed] (-0.5,0) -- (0.6,4);
  \node at (0.8,4) [above] {$\frac{x-\bar\xi}{g^\prime(u_{\mathrm{free}}-)}$};
  \node at (-2,1) [above] {\color{blue}$g$};
  \draw[thick, gray] (-0.4,0) -- (0,0.8);
  \draw[thick, gray] (0,0.8) -- (0.4,0);
   \node at (-0.4,0) [below] {$-\delta$};
   \node at (0.4,0) [below] {$\delta$};
\end{tikzpicture}
\begin{tikzpicture}[scale=0.8]
  \draw[->] (-3.4,0) -- (3.4,0) node[right] {$x$};
  \draw[->] (0,-0.5) -- (0,4) node[left] {$t$};
  \draw[thick, blue, dashed] (0,0) -- (-1.5,4);
  \draw[thick, blue, dashed] (0,0) -- (-1,4);
  \draw[thick, blue, dashed] (0,0) -- (-0.5,4);
  \draw[thick, blue, dashed] (0,0) -- (0.5,4);
  \draw[thick, blue, dashed] (0,0) -- (2.4,4);
  \draw[thick, blue, dashed] (0,0) -- (3,4);
   \draw[thick, blue, dashed] (0,0) -- (3.6,4);
   \node at (3.9,4) [above] {$\frac{x}{g^\prime(0)}$};
  \draw[very thick, black] (0,0) -- (2,4);
  \node at (2.4,4) [above] {$\frac{x}{g^\prime(u_{\mathrm{free}}-)}$};
  \draw[very thick, black] (0,0) -- (0.8,4);
  \node at (0.7,4) [above] {$\frac{x}{g^\prime(u_{\mathrm{free}}+)}$};
  \draw[thick, blue, dashed] (-0.5,0) -- (-2,4);
  \draw[thick, blue, dashed] (-1,0) -- (-2.9,4);
  \draw[thick, blue, dashed] (-1.5,0) -- (-3.8,4);
  \node at (-1,4) [above] {$x/\alpha$};
  \node at (-3,1) [above] {\color{blue}$g$};
  \draw[thick, gray] (-0.4,0) -- (0,0.8);
  \draw[thick, gray] (0,0.8) -- (0.4,0);
   \node at (-0.4,0) [below] {$-\delta$};
   \node at (0.4,0) [below] {$\delta$};
\end{tikzpicture}
\caption{Two-Phase Model:  Case 3. Representation of the characteristic curves of the rarefaction $u^\flat(t,x)$, when $u^-<u_{\mathrm{free}}<u^+<\widehat u_{\mathrm{free}}$ on the left, and when $u^-\geq u_{\mathrm{free}}$ and $u_{\mathrm{free}}<u^+<\widehat u_{\mathrm{free}}$ on the right.}\label{fig:curves3fr}
\end{figure}
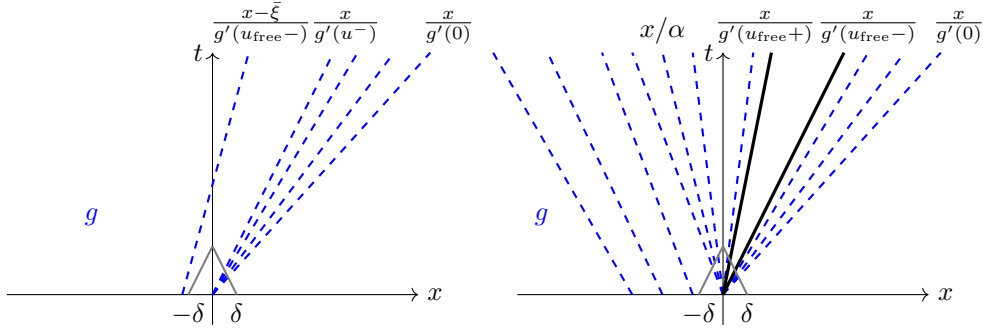

    \item $u^- \geq u_{\mathrm{free}}$. 
    Recall from \eqref{eq:sol3} that the interface in the solution of the Riemann problem moves with the speed $\lambda$, corresponding to a jump from $u_{\mathrm{free}}$ to $u^+$. Since $g(u_{\mathrm{free}}) = f(u_{\mathrm{free}})$, this speed is given by
    $$
        \lambda = \frac{f(u^+)-g(u_{\mathrm{free}})\equiv f(u_{\mathrm{free}})}{u^+-u_{\mathrm{free}}}.
    $$
    Given that $u_{\mathrm{free}} < u^+ < \widehat{u}_{\mathrm{free}}$, the definition of $\widehat u_{\mathrm{free}}$ and the concavity of $f$ implies
    $$
    g'(u_{\mathrm{free}}+)=\frac{f(\widehat{u}_{\mathrm{free}})-f(u_{\mathrm{free}})}{\widehat{u}_{\mathrm{free}}-u_{\mathrm{free}}}<\frac{f(u^+)-f(u_{\mathrm{free}})}{\widehat{u}_{\mathrm{free}}-u_{\mathrm{free}}}<f'(u^+).
    $$
    Consequently, we obtain the chain of inequalities
    $$
    g'(u_{\mathrm{free}}+) < f'(u^+)<\lambda<f'(u_{\mathrm{free}})=g'(u_{\mathrm{free}}-).
    $$
    By construction, the initial data $u(0, x)$ in \eqref{eq:sys_ubC3} and \eqref{eq:sys_udC3} satisfy $\bar u(0-)=u_{\mathrm{free}}$ and $\bar u(0+)=u^+$. For every $\epsilon>0$, there exists $\delta>0$ such that
    \begin{align}
        |u^\flat(t,x) - u_{\mathrm{free}}| < \epsilon &\quad \text{whenever } (x,t)\in \{t+|x| < \delta\}\cap \mathcal{D},
        \\
        |u^\sharp(t,x) - u^+| < \epsilon & \quad\text{for } t+|x| < \delta,
    \end{align}
    where, the region $\mathcal{D}$ is defined according to the sign of the parameter $\alpha = g'(\min\{u^-,u^+\})<g'(u_{\mathrm{free}}+)$: if $\alpha <0 $ then $\mathcal{D}=\{t>x/\alpha, x<0\}\cup\{t>0, x>0\}$; if $\alpha >0 $ then $\mathcal{D}=\{0<t<x/\alpha, x>0\}$ (see Figure \ref{fig:curves3fr}, right).
    \begin{lemma}
        Within the region $\{t+|x| < \delta\}\cap \mathcal{D}$, it holds 
        $$
        H(t,x)=\frac{g(u^\flat(t,x-))-f(u^\sharp(t,x+))}{u^\flat(t,x-) - u^\sharp(t,x+)}> g'(u_{\mathrm{free}}+).
        $$
    \end{lemma}
    \begin{proof}
        First, assume that $t \leq x / g'(u_{\mathrm{free}}+)$. In this region, by construction, $u^\sharp(t,x)$ is close to $u^+$ and satisfies $u^+ \leq u^\sharp(t,x) < \widehat{u}_{\mathrm{free}}$. Moreover, $u^\flat(t,x) \leq u_{\mathrm{free}}$, which implies $g(u^\flat) \equiv f(u^\flat)$.
        Then, by the monotonicity property of secant slopes for concave functions,
        $$
        H(t,x)=\frac{f(u^\flat(t,x-))-f(u^\sharp(t,x+))}{u^\flat(t,x-) - u^\sharp(t,x+)} > \frac{f(\widehat  u_{\mathrm{free}})-f(u^\flat)}{\widehat u_{\mathrm{free}} - u^\flat}> \frac{f(\widehat  u_{\mathrm{free}})-f(u_{\mathrm{free}})}{\widehat u_{\mathrm{free}} - u_{\mathrm{free}}}=g'(u_{\mathrm{free}}+).
        $$
        Next, assume $t > x / g'(u_{\mathrm{free}}+)$. Now, $u_\mathrm{free}<u^\flat<u^+<u^\sharp<\widehat u_\mathrm{free}$ and, by the concavity of $g$
        $$
        g(u^\flat)\geq g(u_\mathrm{free})+g'(u_\mathrm{free}+)(u^\flat-u_\mathrm{free}) = f(u_\mathrm{free})+g'(u_\mathrm{free}+)(u^\flat-u_\mathrm{free}).
        $$
        Therefore,
        \begin{align*}
            H(t,x)=\frac{f(u^\sharp(t,x+))-g(u^\flat(t,x-))}{ u^\sharp(t,x+) - u^\flat(t,x-)} &\geq \frac{f(u^\sharp)-f(u_\mathrm{free})}{u^\sharp - u^\flat}-g'(u_\mathrm{free}+)\frac{u^\flat-u_\mathrm{free}}{u^\sharp - u^\flat}
            \\
            &>g'(u_\mathrm{free}+)\frac{\widehat u_\mathrm{free}-u_\mathrm{free}}{u^\sharp - u^\flat}-g'(u_\mathrm{free}+)\frac{u^\flat-u_\mathrm{free}}{u^\sharp - u^\flat}
            \\
            &>g'(u_\mathrm{free}+).
        \end{align*}
    \end{proof}
    By the previous Lemma, any integral curve of \eqref{eq:Yode3fr}, starting from the origin immediately enters the region where $x>g'(u_\mathrm{free}+) t$. Indeed, if $y(0)=0$, then 
    $$
    \frac{d}{dt}(y(t)-g'(u_\mathrm{free}+)t)=H(t,y(t))-g'(u_\mathrm{free}+)>0,
    $$
    and therefore $y(t)>g'(u_\mathrm{free}+)t$ for sufficiently small $t>0$. By  construction, we have $u^\flat(t, y(t))<u_\mathrm{free}$. Since $f\equiv g$ on $[0, u_\mathrm{free}]$, $H(t, x)$ reads \eqref{eq:H3fr}. Moreover, $u^\sharp(t,y+)-u^\flat(t,y-)\geq u^+-u_\mathrm{free}>0$, so the denominator remains uniformly bounded away from zero. Therefore, $H(t, x)$ is locally Lipschitz in this region. Hence, the Cauchy problem \eqref{eq:Yode3fr} admits a unique local solution.
\end{itemize}

\begin{figure}[h!]
\centering
\begin{tikzpicture}[scale=1.0]
  \draw[->] (-0.5,0) -- (5,0) node[right] {};
  \draw[->] (0,-0.5) -- (0,3.5) node[above] {};
  \draw[domain=0:4, smooth, variable=\x, red, thick] 
       plot ({\x}, {0.8*\x*(4-\x)});
  \draw[domain=0:0.8, smooth, variable=\x, blue, thick] 
       plot ({\x}, {0.8*\x*(4-\x)});    
  \draw[domain=0.8:4, smooth, variable=\x, blue, thick] 
       plot ({\x}, {0.5*(\x+0.455)*(4-\x)});
  \node at (-0.2,0) [below] {$0$};
  \node at (4,0) [below] {$u_{max}$};
  \draw[thick, black, dashed] (0.8,0) -- (0.8,2);
  \node at (0.8,0) [below] {$u_{\mathrm{free}}$};
   \draw[thick, black] (0.8,{0.8*0.8*(4-0.8)}) -- (2.1,{0.8*2.1*(4-2.1)});
   \draw[thick, black] (0.2,1.04) -- (1.5,3.4);
   \node at (1.5,3.4) [above] {$g'(u_{\mathrm{free}}-)$};
   \node at (2.8,3.1) [above] {$g'(u_{\mathrm{free}}+)$};
   \node at (2.2,0) [below] {$\widehat{u}_{\mathrm{free}}$};
   \draw[thick, black, dashed] (2.1,0) -- (2.1,{0.8*2.1*(4-2.1)});
   \node at (1.5,0) [below] {$u^+$};
   \draw[thick, black, dashed] (1.5,0) -- (1.5,{0.8*1.5*(4-1.5)});
   \node at (0.25,0) [below] {$u^-$};
   \draw[thick, black, dashed] (0.2,0) -- (0.2,{0.8*0.2*(4-0.2)});
   \draw[very thick, green] (0.2,{0.8*0.2*(4-0.2)}) -- (1.5,{0.8*1.5*(4-1.5)});
\end{tikzpicture}
\begin{tikzpicture}[scale=1.0]
  \draw[->] (-0.5,0) -- (5,0) node[right] {};
  \draw[->] (0,-0.5) -- (0,3.5) node[above] {};
  \draw[domain=0:4, smooth, variable=\x, red, thick] 
       plot ({\x}, {0.8*\x*(4-\x)});
  \draw[domain=0:0.8, smooth, variable=\x, blue, thick] 
       plot ({\x}, {0.8*\x*(4-\x)});    
  \draw[domain=0.8:4, smooth, variable=\x, blue, thick] 
       plot ({\x}, {0.5*(\x+0.455)*(4-\x)});
  \node at (-0.2,0) [below] {$0$};
  \node at (4,0) [below] {$u_{max}$};
  \draw[thick, black, dashed] (0.8,0) -- (0.8,2);
  \node at (0.8,0) [below] {$u_{\mathrm{free}}$};
   \draw[thick, black] (0.8,{0.8*0.8*(4-0.8)}) -- (2.1,{0.8*2.1*(4-2.1)});
   \draw[thick, black] (0.2,1.04) -- (1.5,3.4);
   \node at (1.5,3.4) [above] {$g'(u_{\mathrm{free}}-)$};
   \node at (2.8,3.1) [above] {$g'(u_{\mathrm{free}}+)$};
   \node at (2.2,0) [below] {$\widehat{u}_{\mathrm{free}}$};
   \draw[thick, black, dashed] (2.1,0) -- (2.1,{0.8*2.1*(4-2.1)});
   \node at (1.5,0) [below] {$u^+$};
   \draw[thick, black, dashed] (1.5,0) -- (1.5,{0.8*1.5*(4-1.5)});
   \draw[very thick, green] (0.8,{0.8*0.8*(4-0.8)}) -- (1.5,{0.8*1.5*(4-1.5)});
    \node at (3,0) [below] {$u^-$};
    \draw[thick, black, dashed] (3,0) -- (3,{0.5*(3+0.455)*(4-3});
\end{tikzpicture}
\caption{Two-Phase Model:  Case 3. Representation of $\dot y(0+)=\lambda$ (in green) when $u^-<u_{\mathrm{free}}<u^+<\widehat u_{\mathrm{free}}$ on the left, and when $u^-\geq u_{\mathrm{free}}$ and $u_{\mathrm{free}}<u^+<\widehat u_{\mathrm{free}}$ on the right.}
\end{figure}
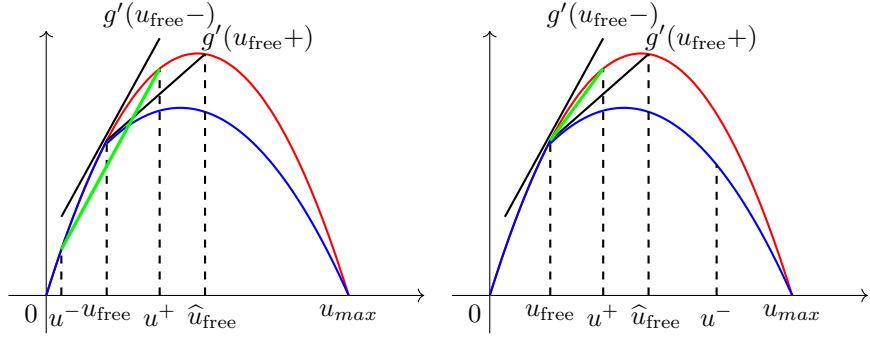

\bibliographystyle{plain}
\bibliography{biblio}
\end{document}